\documentclass[11pt]{amsart}
\usepackage{amsthm,amsfonts,graphicx,hyperref,tikz,enumerate,psfrag,bm}
\usepackage{amssymb}
\usepackage{fullpage}

\usepackage[utf8]{inputenc} 

\usepackage{stmaryrd}
\RequirePackage{amsthm,amsmath,amsfonts,tikz}

\newtheorem{lemma}{Lemma}
\newtheorem{remark}{Remark}[section]
\newtheorem{proposition}[lemma]{Proposition}
\newtheorem{theorem}[lemma]{Theorem}

\newtheorem{definition}[lemma]{Definition}

\newtheorem{theorema}{Theorem}

\newcommand{\bbE}{\mathbb{E}}

\newcommand{\bbN}{\mathbb{N}}

\newcommand{\bbP}{\mathbb{P}}

\newcommand{\bbR}{\mathbb{R}}

\newcommand{\bbZ}{\mathbb{Z}}

\newcommand{\dE}{\mathbb {E}}
\newcommand{\dP}{\mathbb{P}}
\newcommand{\dN}{\mathbb {N}}
\newcommand{\dZ}{\mathbb {Z}}
\newcommand{\dR}{\mathbb {R}}

\newcommand{\cL}{\mathcal{L}}
\newcommand{\cI}{\mathcal{I}}
\newcommand{\cA}{\mathcal{A}}
\newcommand{\cC}{\mathcal{C}}
\newcommand{\cF}{\mathcal{F}}
\newcommand{\cB}{\mathcal{B}}

\newcommand{\cP}{\mathcal {P}}
\newcommand{\cY}{\mathcal {Y}}
\newcommand{\cX}{\mathcal {X}}
\newcommand{\cT}{\mathcal {T}}

\newcommand{\cG}{\mathcal {G}}
\newcommand{\cR}{\mathcal {R}}
\newcommand{\cU}{{U}}
\newcommand{\cQ}{\mathcal {Q}}
\newcommand{\cS}{\mathcal {S}}

\newcommand{\gep}{\varepsilon}       % \ge already exists...

\newcommand{\gl}{\lambda}

\newcommand{\ind}{\mathbf{1}}

\newcommand{\fh}{ \mathfrak{h}}

\newcommand{\bP}{\mathbf{P}}

\newcommand{\bp}{\mathbf{p}}
\newcommand{\bq}{\mathbf{q}}
\newcommand{\ba}{\mathbf{a}}
\newcommand{\bb}{\mathbf{b}}

\newcommand{\BRA}[1]{{{\left\{#1\right\}}}} % {1}
\newcommand{\NRM}[1]{{{\left\| #1\right\|}}} % ||1||
\newcommand{\PAR}[1]{{{\left(#1\right)}}} % (1)
\newcommand{\1}{1\!\!{\sf I}}\newcommand{\IND}{\1}
\newcommand{\veps}{\varepsilon}

\newcommand{\Diam}{\mathrm {diam} }
\newcommand{\Sp}{\mathrm{Sp} }

\newcommand{\AND}{\quad \hbox{ and } \quad}
\newcommand{\NB}{\mathbf{NB}}
\newcommand{\Tm}{T^{\mathrm{mix}}}
\newcommand{\Tmn}{T^{\mathrm{mix}}_n}

\newcommand{\D}{\mathrm{Dist}}
\newcommand{\TV}{\mathrm{TV}}

\definecolor{darkred}{rgb}{0.7,0.1,0.1}

\definecolor{darkblue}{rgb}{0.1,0.1,0.7}

\renewcommand{\rho}{\varrho}
\newcommand{\limsuptwo}[2]{\limsup_{\substack{#1 \\ #2}}}     % \limsup with 2 lines

 \title[Cutoff at the entropic time]{Cutoff at the entropic time for random walks on covered expander graphs}

\author{Charles Bordenave and Hubert Lacoin}
\begin{document}

\begin{abstract}
It is a fact simple to establish that the mixing time of the simple random walk on a $d$-regular graph $G_n$ with $n$ vertices is asymptotically bounded from below by 
$\frac{d  }{d-2 } \frac{\log n}{\log (d-1)}$. Such a bound is obtained 
by comparing the walk on $G_n$ to the walk on $d$-regular tree $\cT_d$.
If one can map another transitive graph $\cG$ onto $G_n$, then we can improve the strategy by using a comparison with the random walk on $\cG$ (instead of that of $\cT_d$), and we obtain a lower bound of the form $\frac{1}{\fh}\log n$, where $\fh$ is the entropy rate associated with $\cG$. We call  this  the entropic lower bound.\\
It was recently proved that in the case $\cG=\cT_d$, this entropic lower bound (in that case $\frac{d  }{d-2 } \frac{\log n}{\log (d-1)}$) is sharp when graphs have minimal spectral radius and thus that in that case the random walk exhibit cutoff at the entropic time.
In this paper, we provide a generalization of the result by providing a sufficient condition on the spectra of the random walks on $G_n$ under which the random walk exhibits cutoff at the entropic time. It applies notably to anisotropic random walks on random $d$-regular graphs and to random walks on random $n$-lifts of a base graph (including non-reversible walks).
\\[10pt]
  \\[10pt]
\end{abstract}

\maketitle

\section{Introduction}

This paper is aimed at understanding the mixing properties of random walks on a finite regular graph. We are going to be focused on asymptotic properties when the number of vertices goes to infinity.

\medskip

\subsection*{Minimal mixing time for the simple random walk.}
Let $3 \leq d \leq n-1$ be integers with $nd$ even and let $G_n = (V_n,E_n) $ be a finite simple $d$-regular graph on a vertex set $V_n$ of size $\# V_n = n$. Let $(X_t)_{t\ge 0}$  be the {\em simple random walk}
on $G_n$, which is the Markov process taking values in $V_n$ with transition matrix, 
\begin{equation*}\label{transitionG}
 P_n(x,y)= \frac{\ind_{\{\{x,y\} \in E_n\}}}{d} \quad \quad  \text{for }x,y\in V_n.
\end{equation*}
The uniform measure  on $V_n$ denoted by $\pi_n$ is reversible for the process. Furthermore if $G_n$ is connected, then $\pi_n$ is the unique invariant probability measure of $P_n$. If additionally  $G_n$ is  not bipartite, then $P_n^t(x,\cdot)$
 converges to $\pi_n$ when $t$ tends to infinity.

\medskip

We are interested in estimating the time at which $P_n^t(x,\cdot)$ falls in a close neighborhood of $\pi_n$ in terms of the {\em total variation} distance. More formally,  the {\em total variation mixing time}  associated with threshold  $\gep\in(0,1)$ and initial condition $x \in V_n$, is defined by 
\begin{equation*}
\Tmn (x,\gep) :=\inf\left\{ t  \in \dN \ : \ d_n(x,t)<\gep \right\},
\end{equation*}
where  $d_n(x,t)$ is the {\em total variation} distance to equilibrium
\begin{equation}\label{eqdist}
 d_n(x,t):=  \| P_n^t(x,\cdot)-\pi_n\|_{\TV} =  \frac 1 2 \sum_{y \in V_n} \left| P_n^t(x,y)-\pi_n(y) \right|=
 \max_{A\subset V_n} \left\{  P_n^t(x,A)-\pi_n(A) \right\}.
\end{equation}
The {\em worst-case mixing time} is classically defined as 
$$
\Tmn (\veps ) = \max_{x \in V_n} \Tmn(x ,\veps).
$$

The mixing properties for the random walk are intimately related to the spectrum of $P_n$. An illustration of this is the classical computation based on the spectral decomposition of $P_n$ (see \cite[Theorem 12.4]{MR3726904} for a proof in the reversible case) which allows to control the distance as a function of the {\em singular radius} of $P_n$.  For all $x \in V_n$, 
\begin{equation}\label{standspec}
  d_n(x,t)\le \frac{\sqrt{n-1}}{2}\sigma_n^t.
\end{equation}
where the singular radius $\sigma_n$ 
$$\sigma_n = \| (P_n )_{|\ind^\perp}\|_{2\to 2} $$
 is the $\ell_2$ operator norm of $P_n$  restricted to functions with zero sum. Since $P_n$ is reversible, we have $\sigma_n = \rho_n$ where $\rho_n$ is the spectral radius of $P_n$, that is the second largest eigenvalue of $P_n$ in absolute value counting multiplicities. This yields in particular that 
\begin{equation}\label{basicupper}
\Tmn(\gep)\le   \frac{1}{| \log \rho_n|} \left( \frac{1}{2}\log n -   \log (2\veps)\right).
\end{equation}
In particular, if we have $\rho_n < 1-\delta$ for some fixed $\delta \in (0,1)$ along some sequence of integers  going to infinity, then the upper bound  in \eqref{basicupper} is of order $\log n$ along that sequence.

\medskip

On the other hand, a naive lower bound of the same order of $\Tmn(\gep)$ can be obtaind by using the elementary fact that the graph distance  $\D(x,X_t)$ between $X_t$ and the initial condition $x$ is stochastically dominated by a random walk on the set of non-negative integers, started at $0$, with jump probabilities $1/d$ to the left and $(d-1)/d$ to the right, except at $0$ where the probability to jump to the right is equal to $1$. Thus when starting from $X_0=x$,  $X_t$ remains within distance $r$ from  $x$    at least during a random time of order $\frac{d}{d-2}r +O(\sqrt{r})$. Combining this with the fact that a ball of radius $r$ contains at most $d(d-1)^{r-1}$ vertices, we obtain that  for any $x \in V_n$
\begin{equation}\label{basiclower}
\Tmn(x,1-\gep) \ge    \frac{d}{(d-2) \log (d-1)}\left(\log n- C_{\gep}\sqrt{\log n} \right).
\end{equation}
While the strategy might seem a bit rough, the above bound \eqref{basiclower} can be sharp. This was first discovered for random $d$-regular graphs in \cite{MR2667423}.

\medskip

However, an important observation is that the factor in front of $\log n$ in \eqref{basicupper} and \eqref{basiclower} cannot match. From Alon-Boppana lower bound \cite{MR875835,MR2679612}, we have for any sequence of $d$-regular graphs $(G_n)_{n\ge 0}$
on $n$ vertices, we have $\liminf_n\rho_n \geq \rho:= \frac{2\sqrt{d-1}}{d}$. More precisely, there exists a constant $C = C(d)$ such that for every $n$ and every $d$-regular graph on $n$ vertices
\begin{equation}\label{eq:AB}
\rho_n  \ge \rho-\frac{C}{(\log n)^2}.
\end{equation}
The number $\rho=  2\sqrt{d-1} / d$ is the spectral radius of the simple random walk on the infinite $d$-regular  tree $\cT_d$ (and incidentally also that   of the biased random walk on integers used in the lower-bound strategy). A graph such that $\rho_n \leq \rho$ is called a (non-bipartite) {\em Ramanujan graph}.
Hence a natural question is the following: 
\textit{If a sequence of graphs on $n$ vertices has an asymptotically minimal spectral radius in the sense that $\rho_n=(1 + o(1))\rho $, does it also have a minimal mixing time in the sense that $\Tmn(\gep) = (1+ o(1))  \frac{(d-2)}{d \log (d-1)} \log n$ for any fixed $\veps \in (0,1)$?}

\medskip

\noindent An affirmative  answer was given to this question in \cite{MR3558308} (see also \cite{MR3693771}):

\begin{theorema}[\cite{MR3558308}]\label{th:LP}
Let $d \geq 3$ be an integer and let  $(G_n)_{n\in \bbN}$ be a sequence of $d$-regular graphs on $n$ vertices, for which the associated sequence of spectral radii satisfy $\lim_{n \to \infty} \rho_n = \rho = 2\sqrt{d-1} / d$. Then  for any $\gep\in (0,1)$, we have
\begin{equation}\label{kutof}
\lim_{n\to \infty}\frac{\Tmn(\gep)}{\log n} =\frac{d}{(d-2)\log (d-1)}.
\end{equation}
\end{theorema}

\begin{remark}
The result above remains of course valid if our sequence $(G_n)$ is indexed by an infinite subset of $\bbN$  provided that $\rho_n$ converges to $\rho$ when $n\to \infty$ in this subset. In the remainder of the paper, with a some small abuse of notation, when using $\lim$, we always assume that the considered sequence may not be defined for every $n$.
\end{remark}

Theorem \ref{th:LP} is an illustration of the {\em cutoff phenomenon}.  A sequence of finite Markov chains corresponding to the sequence of transition matrices $(P_n)$ exhibits cutoff if up to first order in $n$, the mixing time $\Tmn(\veps)$ does not depend on $\veps\in (0,1)$, that is, for any $\veps \in (0,1)$, $\lim_{n \to \infty} \Tmn(\gep)/ \Tmn(1-\gep) =1$. Since its original  discovery  by  Diaconis,  Shashahani  and  Aldous  in the context  of  card  shuffling  \cite{aldous1983mixing,aldous1986shuffling, diaconis1981generating}, this phenomenon has attracted much attention. We refer to  \cite{diaconis1996cutoff,MR3726904} for an introduction and to \cite{MR3650406} for an alternative characterization of cutoff.
For other recent contributions on cutoff for random walks on graphs with bounded degrees, see \cite{MR3650414,MR3758735,MR3773804}.

\medskip

As a warmup, we provide a novel proof of Theorem \ref{th:LP}, which is simpler than those presented in \cite{MR3558308} and \cite{MR3693771} (independently observed by Lubetzky \cite{lubetzky}).  A more precise version of Theorem \ref{th:LP} will be proved in Proposition \ref{kant} below (it notably allows to obtain the second order term in the asymptotic development of $\Tmn(\gep)$). With our approach we can also relax the assumption by allowing the presence of $n^\alpha$ eigenvalues at a positive distance from the interval $[-\rho,\rho]$, with $\alpha \in (0,1)$ small enough, at the cost of discarding a small proportion of possibly bad starting points (the methods in \cite{MR3693771,MR3558308} only allow for $n^{o(1)}$ outlying eigenvalues, see remark below  \cite[Corollary 5]{MR3558308}). More precisely, given  $(G_n)$ a sequence of $d$-regular graphs on $n$ vertices, we define the upper semi-continuous function $I : [0,1] \to \{-\infty\} \cup [0,1]$,
which can be interpreted as an asymptotic density of eigenvalues on $\log$-$\log$ scale
\begin{equation}\label{azymp}
I(u) = \inf_{\veps \downarrow 0} \limsup_{n\to \infty} \frac{ \log \left( \sum_{\{ \gl \in \Sp (P_n) \, : \,  
||\gl|-u|<\gep\}} \mathrm{dim}(E^{\gl}_n)\right)}  {\log n}, 
\end{equation}
where $\mathrm{dim}(E^{\gl}_n)$ denotes the dimension of the eigenspace corresponding to $\gl$.

\begin{theorema}\label{goodold}
Let $\delta \in (0,1)$, $d \geq 3$ an integer and let $(G_n)$ be a sequence of  $d$-regular graphs on $n$ vertices whose  spectral radii satisfy for all $n$, $\rho_n \leq 1 - \delta$ and  for all $ u > \rho$,
 \begin{equation}\label{supoz}
 I(u) \leq 1 - 2\frac{ \log (u/\rho + \sqrt{(u/\rho) ^2 - 1})  }{ \log (d-1)}.
 \end{equation}
Then, there exists $c = c(\delta,d) >0$ such that for any $\gep\in (0,1)$ and $\eta>0$, 
\begin{equation}\label{exceptafew}
\lim_{n \to \infty}  \# \left\{x \in V_n \ : \frac{\Tmn (x,\gep)}{\log n}\ge (1+\eta)  \frac{d }{(d-2)\log (d-1)} \right\} / n^{ 1 - c \eta} = 0.
\end{equation}

\end{theorema}

We note that if the graph $G_n$ is transitive (that is for any pair $x,y \in V_n$, there exists an automorphism of $G_n$ which maps $x$ to $y$) then $\Tmn(x,\veps)$ does not depend on $x$,
and \eqref{exceptafew} implies that $\lim_{n\to \infty} \Tmn (\gep) / \log n =d / ((d-2)\log (d-1))$.  See Remark \ref{rq:flat0} for a variant of Theorem \ref{goodold} which allows to control $\Tmn(\gep)$  at the cost of modifying the definition of the function $I(u)$. 
The principal aim of this paper is to obtain a better understanding of this phenomenon via bringing the question to a larger setup.

\subsection*{Minimal mixing time for the anisotropic random walk.} 

A first possible extension is to consider a  random walk on $G_n$ with non-uniform jump rates. For $d \in \dN$, we set $[d] = \{1,\dots, d\}$.  One way to define a biaised random walk on  a $d$-regular graph $G_n = (V_n,E_n)$ with $ \# V_n=n$ is to assume that $E_n$ can be partitioned into $d$ sets of edges $(E_{n,i})_{i\in [d]}$ where each vertex of $V_n$ is adjacent to exactly one edge of each type (this implies in particular that $n$ is even), and to associate a transition rate $p_i$ to each type of edge with $\sum_{i\in [d]} p_i=1$.
For more generality, we consider  an involution $*  :  i \mapsto i^*$  of $[d]=\{1,\dots, d\}$. We are going to make the weaker assumption that $G_n$ is a  {\em Schreier graph}. 
This means that its adjacency matrix $P_n$ may be written as a sum of permutation matrices. That is to say that for all $x,y \in V_n$, we have
\begin{equation}\label{eq:adjG}
 P_n(x,y)=\sum_{i=1}^d S_i(x,y),
\end{equation}
where, for every $i \in [d]$, $S_i(x,y) = \IND( x = \alpha_i( y))$ for some permutation $\alpha_i$ on $V_n$ and  $S_{i*}= S^{-1}_i$.
In full generality the expression \eqref{eq:adjG}
allows for both $P_n(x,y)\ge 2$, and $P_n(x,x)\ge 1$, so that we consider graphs which may include loops and/or multiple edges. For example, if the involution $ *$ on $[d]$ is the identity, then the permutations $\alpha_i$ are involutions: for every $i \in [d]$, we have $\alpha_i ^{-1} = \alpha_{i^*} = \alpha_i$. In this case, the sets $(E_{n,i})_{i \in [d]}$ defined for every $i \in [d]$ by $E_{n,i} = \{ \{ x, \alpha_i(x)\} :  x \in V_n \}$ is a partition of the edge set $E_n$. We thus recover the above setting.

If $d$ is even, any finite $d$-regular graph is a Schreier graph  for some collection of $d/2$ permutations and their inverses
(another formulation of this result is: \textsl{any $2k$-regular graph is 2-factorable} see \cite{MR1554815}, this is now a standard exercise in graph theory and can be proved using K\"onig's Theorem for bipartite graphs, see e.g.\ \cite[Theorem 6.2.4]{MR2536865}). 

\medskip

This definition of Schreier graphs can be extended to regular graphs on countably many vertices.
Note that any {\em Cayley graph} of a finitely generated group with a symmetric set of generators of size $d$ is a Schreier graph: the natural choice for the permutations $S_i$  in \eqref{eq:adjG} corresponds to the (left or right)  multiplication by an element of the symmetric set of generators, the involution maps a generator to its inverse.

\medskip

Now we consider  $G_n$ is a $d$-regular Schreier graph with $\# V_n = n$, 
given with an involution $*$ and a decompotion of the adjacency matrix into permutations \eqref{eq:adjG},
and  $\bp=(p_1,\dots,p_d)$  a probability vector (that is, a vector whose coordinate are nonnegative and sum to one) 
we define the matrix  
\begin{equation}\label{asym}
P_{n,\bp} = \sum_{i=1}^d p_i S_i.
\end{equation}
Note that by construction $P_{n,\bp}$ is a stochastic matrix. This is the transition kernel of a random walk on $G_n$ which we refer to as the {\em {\bf p}-anisotropic random walk}. Again,  $\pi_n$, the uniform measure on $V_n$, is invariant for this process.  We are going to assume that 
\begin{equation}\label{hippo}
d\ge 3 \quad \text{ and } \quad  \forall i\in [d],   \  p_i + p_{i^*} >0.
\end{equation}
The condition $p_i + p_{i^*} >0$ is not really a restriction since it can be satisfied by just eliminating the coordinates for which $p_i + p_{i^*}=0$. The condition $d\ge 3$ (which is not the same as asking that three coordinates of $\bp$ are positive) is very natural and  
justified below Equation \eqref{lalowertG}. The singular radius of $P_{n,\bp}$ is defined as the $\ell_2$ operator norm of $P_{n,\bp}$ projected onto the orthogonal of constant functions 
 \begin{equation}\label{eq:specrad}
\sigma_{n,\bp} = \| (P_{n,\bp} )_{\ind^\perp} \|_{2 \to 2}.
\end{equation} 
Recall that the {\em singular values} of a matrix $T$ are the square of the eigenvalues of $T T^*$. By definition $\sigma_{n,\bp}$ is the second largest singular value of $P_{n,\bp}$ (we are counting eigenvalues with multiplicities, meaning that $\sigma_{n,\bp}=1$ for a non-connected graph).
If $P_{n,\bp}$ is reversible then $\sigma_{n,\bp}$ coincides with the spectral radius $\rho_{n,\bp}$, that is the second largest eigenvalue of $P_{n,\bp}$ in absolute value. 
Note that $P_{n,\bp}$ is reversible if the following condition holds:
\begin{equation}\label{eq:sympi}
\forall i\in [d], \quad p_{i^*} = p_i.
\end{equation}
Our aim is to prove a result analogous to Theorem \ref{th:LP}
for $\bp$-anisotropic walks on  Schreier graphs. We fix the involution $*$ and $\bp$ and then investigate the asymptotic behavior of the mixing time for $\bp$-anisotropic random walks on a sequence of Schreier graphs $(G_n)$ associated with the involution $*$.

\medskip

Instead of comparing the spectral radius of $P_n$ with that of the simple random walk on the $d$-regular tree, we need here to compare it with that of a $\bp$-anisotropic walk on the tree $\cT_d$ considered as a Cayley graph.
There are several natural ways to endow $\cT_d$ with a Cayley graph structure. For instance, we can consider $k$ free copies of $\dZ/2\dZ$  and $l$ free  copies of $\dZ$ with their natural generators, for any value of  $k$ and $l$ satisfying $k+2l=d$. 
We are going to choose $k$ to be equal to the number of fixed points of $*$ so that the infinite object we consider has a structure which is analogous to our finite Schreier graphs (with Definition \ref{def:covering} below we will formalize this remark).

\medskip

Using the Schreier graph structure of $\cT_d$ considered as a Cayley graph, we define in a manner analogous to \eqref{asym} the $\bp$-anisotropic random walk on $\cT_d$. We denote by $\cP_{\bp}$ its transition kernel.
These random walk have been extensively studied in the literature (see e.g.\ \cite{MR1219707, MR824708, MR1832436}).

\medskip

In analogy with \eqref{eq:AB}, in the reversible case where \eqref{eq:sympi} holds, one can asymptotically compare the spectral radius of $P_{n,\bp}$ with that of $\cP_{\bp}$.
From \cite{CECCHERINISILBERSTEIN2004735, MR1802431}, the Alon-Bopanna lower bound for the 
spectral radius states that for any sequence of Schreier graphs we have
\begin{equation}\label{lalower0}
 \liminf_{n\to \infty}\rho_{n,\bp}\ge \rho_\bp ,
\end{equation}
where $\rho_\bp$  is the spectral radius of $\cP_{\bp}$, given by the classical Akemann-Ostrand formula \cite{MR0442698}.
In the general case, a lower bound of this type holds for the singular radii of powers of $P_{n,\bp}$. More precisely, for integer $t \geq 1$, we define the {\em $t$-th singular radius} as 
\begin{equation}\label{eq:defsigt}
\sigma_{n,\bp} (t) = \| (P^t_{n,\bp} )_{\ind^\perp} \|^{1/t}_{2 \to 2} \quad \hbox{ and } \quad \sigma_{\bp} (t)  = \| (\cP^t_{\bp} ) \|^{1/t}_{2 \to 2}.
\end{equation}
We simply write $\sigma_{n,\bp}$ and $\sigma_{\bp}$ when $t=1$.
Moreover, Gelfand's formula asserts that the $t$-th singular radius converges to the  spectral radius,
\begin{equation}\label{eq:gelfand}
\lim_{t \to \infty} \sigma_{n,\bp} (t) =   \rho_{n,\bp} \quad \hbox{ and } \quad \lim_{t \to \infty} \sigma_{\bp} (t) =   \rho_{\bp}.
\end{equation}

 \medskip
Note that if  \eqref{eq:sympi} holds then for any $t \geq 1$, $\sigma_{n,\bp} (t) = \rho_{n,\bp}$ and $\sigma_{\bp} (t) = \rho_{\bp}$. 
Beware here and througout this text that the spectral radius $\rho_{\bp}$ is the spectral radius of the bounded operator $\cP_{\bp}$ on $\ell^2(\cG)$. It can differ (in fact it is larger or equal to) from what is often called in the literature the spectral radius of the walk which is the asymptotic rate of decay the return probability, that is  $\lim_{t\to \infty}\cP_{\bp}^{2t} (e,e)^{1/(2t)}$ (in the reversible case, the two notions coincide). From \cite{MR1802431}, the Alon-Bopanna lower bound claims that for any fixed $t$, for any sequence of Schreier graphs we have
\begin{equation}\label{lalowert}
 \liminf_{n\to \infty}\sigma_{n,\bp}(t)\ge \sigma_\bp(t),
\end{equation}
In particular, from Gelfand's formula \eqref{eq:gelfand}, we get
\begin{equation}\label{lalowertG}
\lim_{t \to \infty} \liminf_{n\to \infty}\sigma_{n,\bp}(t)\ge \rho_{\bp}.
\end{equation}
The latter formula can be thought as an extension of \eqref{lalower0} to the non-reversible case.
Note that our assumption \eqref{hippo} simply corresponds to assuming that $\rho_{\bp}<1$  (in the discarded cases, the anisotropic random walk on $\cT_d$ remains on a subset of the tree which is homemorphic to $\bbZ$).

Adapting the reasoning which yields \eqref{basiclower} to the 
isotropic case, we can also obtain an asymptotic lower bound in $n$ for the mixing time for the $\bp$-anisotropic random walk on $G_n$. 
  Consider $(\cX_t)_{t \ge 0}$ an anisotropic random walk on $\cT_d$ with transition kernel $\cP_{\bp}$ and starting from the root of $\cT_d$ denoted by $e$. Introduced by Avez \cite{MR0324741}, the {\em entropy rate} $\fh(\bp)$  of $\cP_{\bp}$ is defined as 
\begin{equation}\label{eq:defhp}
\fh(\bp) := \lim_{t \to \infty}  - \frac{1}{t}\sum_{g\in \cT_d} \cP_{\bp}^t  (e,g) \log \cP_{\bp}^t(e,g).
\end{equation}
We have $\fh(\bp) >0$ as soon as \eqref{hippo} holds. From Shannon-McMillan-Breiman Theorem \cite[Theorem 2.1]{MR704539}, we have almost surely 
 \begin{equation}\label{eq:SMMB}
 \lim_{t \to \infty}   \frac{\log \cP_{\bp}^t (e, \cX_t) }{t} =-\fh (\bp).
 \end{equation}
A way to interpret this convergence is to say that at large times $t$, the marginal distribution of $\cX_t$ is roughly uniform on a (deterministic) set of size $e^{\fh(\bp)t(1+o(1))}$. 
We  have chosen our setup so that we can construct the random walk on $G_n$
by taking the image of $\cX_t$ by some function $\cT_d \to V_n$ (see Definition \ref{def:covering}). Thus for any time $t>0$ and $x \in V_n$, the entropy of $P^t_{n,\bp}(x,\cdot)$ is at most the entropy of $\cP^t_{\bp}(e,\cdot)$ (details are in  Proposition \ref{lb} below). As a consequence, for any fixed $\veps \in (0,1)$, we have
\begin{equation}\label{lentroop}
 \liminf_{t\to \infty} \min_{x\in V_n} \frac{\Tm_{n,\bp}(x,1-\gep)}{\log n}\ge \frac{1}{\fh(\bp)}.
\end{equation}
In the spirit of Theorem \ref{th:LP}, for a given probability vector $\bp$, a natural question is thus the following: {\em If a sequence of graphs on $n$ vertices has a minimal asymptotic spectral radius in the sense that $\lim_{t \to \infty} \limsup_{n\to \infty} \sigma_{n,\bp} (t)= \rho_{\bp}$, does it also have an asympotic minimal mixing time in the sense that  $\lim_{n\to \infty}\frac{\Tmn(\gep)}{\log n} = ( \log n )/ \fh(\bp)  $ for any fixed $\veps \in (0,1)$?}

\medskip

It turns out that in the anisotropic setup, the relation between spectral gap and mixing time could be more subtle. We have an asymptotically minimal mixing time for the $\bp$-anisotropic random walk if the spectral radius is asymptotically minimal 
for another anisotropy vector $\bp'$.

\begin{theorem}\label{aniz}
 Let $d\geq 3$ be an integer, $*$ an involution on $[d]$ and let $\bp$ be a probability vector on $[d]$ which satisfies the condition \eqref{hippo}. Then, there exists 
 another probability vector $\bp'$ with the same support than $\bp$  such that the following holds.  If a sequence of Schreier graphs $G_n$ on $n$ vertices as in \eqref{eq:adjG} satisfies  for all integers $t\geq 1$,
\begin{equation}\label{raminudge}
\lim_{n\to \infty}\sigma_{n,\bp'}(t)= \sigma_{\bp'}(t),
 \end{equation}
  then for every $\gep\in(0,1)$
\begin{equation}\label{anicut}
 \lim_{n\to \infty} \frac{\Tm_{n,\bp}(\gep)}{\log n}=\frac{1}{\fh(\bp)}.
\end{equation}
Finally, if $\bp$ satisfies \eqref{eq:sympi} then $\bp'$ also satisfies \eqref{eq:sympi}.
\end{theorem}
The condition \eqref{raminudge} can be thought as a Ramanujan property for the anisotropic random walk with probability $\bp'$. If \eqref{eq:sympi} holds, condition \eqref{raminudge} is equivalent to $$\lim_{n \to \infty} \rho_{n,\bp'} = \rho_{\bp'}.$$ In the non-reversible case, since $\sigma_{n,\bp} (t) \geq \rho_{n,\bp}$ for any $t \geq 1$, condition \eqref{raminudge} implies that $$\limsup_{n \to \infty} \rho_{n,\bp'} \leq \rho_{\bp'}.$$ Note also that in some cases, this condition \eqref{raminudge} can be relaxed to allow for $n^{o(1)}$ singular values  outside a neighborhood of the interval $[-\rho_{\bp'}, \rho_{\bp'}]$; see Remark \ref{rq:flat1} below. 
An explicit expression for the vector $\bp'$ is provided in the proof. In particular we have that $\bp'=\bp$  in only two cases. The first one is the simple random walk: that is ${\bp}$ is the uniform vector ($p_i=1/d$ for all $i\in [d]$), and our result is thus a generalization of Theorem \ref{th:LP}. 
\noindent The other case is the totally asymetric isotropic walk. It corresponds to the case where $*$ has no fixed point and we have 
$$\forall i\in [d], \quad  p_i p_{i^*}=0 \quad \text{ and } \quad p_i+p_{i*}=\frac{2}{d}.$$
In that case we have   $\rho_{\bp}=\sqrt{2/d}$ (see \cite[Example 5.5]{MR1784419}) and $\fh(\bp)=\log\frac{d}{2}$.  
From Poincar\'e inequality \eqref{eq:poincare} below, Theorem \ref{aniz} is extremely easy to prove in this case.
For $\bp$ different from the uniform vector, a source of example for Theorem \ref{aniz} is in \cite{BC18}.  Up to the involution, we consider independent permutations $\sigma_i$ on $[n]$ vertices: if $i \ne i^*$, $\sigma_i$ is a uniform permutation on $n$ elements and, if $i^* = i$, we take $n$ even and $\sigma_i$ is a uniform  matching on $n$ elements (where a matching is an involution without fixed point). Then, in probability, the condition \eqref{raminudge} is true for any probability vector $\bp'$ which satisfies the condition \eqref{eq:sympi}.

\subsubsection*{A couple of open questions concerning anisotropic random walks}
Let us focus for simplicity on the reversible case \eqref{eq:sympi}. We emphasize again that as soon as $\bp$ is not the uniform vector, the condition \eqref{raminudge} differs from what would be the most natural generalization of the Ramanujan property in the  anisotropic setup. We call this property $\mathfrak{R}(\bp)$: 
\begin{equation}\label{raminudgenaive}
  \lim_{n\to \infty} \rho_{n,\bp}= \rho_{\bp}.
 \end{equation}
We believe that this is not an artifact of our proof and that the result would be false if \eqref{raminudge}  is replaced by \eqref{raminudgenaive}. More precisely we believe that for every $\bp$ which is not the uniform vector there should exists sequences of graph satisfying \eqref{raminudgenaive}
but such that \eqref{anicut} does not hold.
We cannot however prove that $\mathfrak{R}(\bp)$ and 
$\mathfrak{R}(\bp')$ are not equivalent. 
In fact, this question yields two natural open problems:
\begin{itemize}
 \item [(A)] With the setup described above can one find a sequence of Schreier graphs and two probability vectors  $\bp$ and $\bq$   
  satisfying \eqref{eq:sympi} such that $\mathfrak{R}(\bp)$ holds and $\mathfrak{R}(\bq)$ does not?
 \item [(B)] Given $\bp \ne  \bq$ satisfying \eqref{eq:sympi} can one find a sequence of Schreier graphs such that
  $\mathfrak{R}(\bp)$ holds and $\mathfrak{R}(\bq)$ does not?
\end{itemize}
While we believe that the answer to (B) (and hence also to (A)) is positive, to our knowledge, all known examples of graphs satisfying $\mathfrak{R}(\bp)$ for one value of $\bp$ satisfy the property for all values of $\bp$.

\subsection*{Minimal mixing time for covered random walks.}
We now present another extension of Theorem \ref{th:LP}. We start with an extension of the notion of Schreier graph beyond the case of the free group.

\begin{definition}[Group action, covering and Schreier graph]\label{def:covering}
Let $\cG$ be a finitely generated group with unit $e$ and $V$ a finite set. A map $\varphi :  \cG \times V \to V$ is an {\em action of $\cG$ on $V$}, if  we have 
\begin{equation*}
 \forall x\in V, \ \forall g,h \in \cG, \quad  \varphi(e,x) = x \  \text{ and } \ \varphi(gh,x) = \varphi(g,\varphi(h,x))
\end{equation*}
For any $g \in \cG$, we denote by $S_g$ the permutation matrix on $V$ associated to the bijective map on $V$: $x \mapsto \varphi(g,x)$.

\medskip

If $\cA$ is a finite symmetric subset of $\cG$ then the {\em Schreier graph of $(\cG,\cA,\varphi)$} is the graph (with possible loops and multiple edges) on $V$ whose adjacency matrix is 
$
\sum_{g \in \cA} S_g.
$
If $G = (V,E)$ is the Schreier graph of $(\cG,\cA,\varphi)$, we say that $(\cG,\cA)$ is a {\em covering} of $G$. %and the action $\varphi$ is a {\em covering map}.  
\end{definition}

Let us check that this definition of a Schreier graph is equivalent to that given earlier. If the adjacency matrix of a finite graph $G = (V,E)$ is of the form \eqref{eq:adjG}, then $G$ is the Schreier graph of $(\cS_V,\cA,\varphi)$ where $\cS_V$ is the symmetric group on $V$, $\cA = (S_1, \ldots, S_d)$  and the covering map is $\varphi ( \sigma, x) = \sigma(x)$. Conversely, if $G $ is the Schreier graph of $(\cG,\cA,\varphi)$ as in Definition \ref{def:covering} with $\cA = \{ a_1, \ldots, a_d\}$ then its adjacency matrix is of the form \eqref{eq:adjG} where the involution $* : i \mapsto i^*$ is defined as $i^* = j$ if and only if $a_j = a_i^{-1}$. Note that if the involution $*$ on $[d]$ has $q_1$ fixed points and $q_1+q_2$ equivalence classes then $G$ is $d$-regular with $d = q_1 + 2q_2$. As already pointed, the infinite tree $\cT_d$ is the Cayley graph of the group $\cG^{(q_1,q_2)}_{\mathrm{free}}$ generated by $q_1$ free copies  of $\dZ$ and $q_2$ free  copies  of  $\dZ/2\dZ$  with their natural generators denoted $\cA_{\mathrm{free}}$. By considering the group homomorphism from $\cG^{q_1,q_2}_{\mathrm{free}}$ to $\cG$ which maps  $\cA_{\mathrm{free}}$ to $\cA$, we deduce that {\em all Schreier graphs are covered by $(\cG^{q_1,q_2}_{\mathrm{free}},\cA_{\mathrm{free}})$} with the corresponding involution.

\begin{remark}
 Note that if $\bp$ is fixed, the definition of the $\bp$-anisotropic random walk  $\cT_d$ is the same (up to graph isomorphism) for all possible values of $q_1$ and $q_2$. However, the choice of the group structure we endow $\cT_d$ with turns
 out to be of importance when considering coverings. As the groups corresponding to different values of $q_1$ and $q_2$ are not isomorphic, 
 there are $d$-regular graphs $G$ that can be expressed as the Schreier graph for $\cG^{(q_1,q_2)}_{\mathrm{free}}$ (with $q_1+2q_2=d$) for some values of $q_2$ and not for others (more precisely it is harder to find a covering for smaller values of $q_2$).
\end{remark}

The standard example of an action on a finite set is the following. Let $\cG$ be a finitely generated group and $H$ a  subgroup of $\cG$ with a finite index. Then the set  of left cosets $V = \{ g H : g \in \cG\}$ (with $gH = \{ gh  : h \in H\}$) is a finite set and $\varphi$ defined by $\varphi(a,bH) = ab H$ is an action of $\cG$ on $V$.

\medskip

We introduce now a notion of anisotropic walk on a (sequence of) Schreier graph. 
Fixing a group $\cG$ we assume that we have a sequence of finite sets $(V_n)$ with $\# V_n = n$ and $(\varphi_n)$ a sequence of actions of $\cG$ on $V_n$. We consider $\bp$ to be a  probability vector
$\cG$ with finite support. 
We are interested in the random walk on $V_n$ with transition matrix
\begin{equation}\label{eq:defPnp}
P_{n,\bp} = \sum_{g\in \cA} p_g S_g,
\end{equation}
where for $g \in \cG$,  $ S_g$ is as in Definition \ref{def:covering} the permutation matrix associated to the action $\varphi_n$. 
Note that if the support of $\bp$ is contained in a finite symmetric  set $\cA$, then $P_{n,\bp}$ is an anisotropic random walk on  the Schreier graph of $(\cG,\cA,\varphi_n)$. This situation extends the previous setup in both directions: the underlying group is not necessarily the free group and the generating set is not necessarily the natural set of generators. Note that the uniform measure on $V_n$ is always stationnary for this random walk. It is reversible if we assume 
\begin{equation}\label{reversib}
\forall g\in \cG, \ p_g=p_{g^{-1}}.
\end{equation}
We are going to compare the random walk on $V_n$ with a random walk on $\cG$. To this end, we define $\cP_{\bp}$ the transition kernel of this random walk on $\cG$ defined by
 \begin{equation}\label{eq:defPp}
\cP_{\bp}=\sum_{g \in \cA }p_g\lambda(g),
\end{equation}
where, for $g \in \cG$, $\lambda(g)$ is the left multiplication operator in (or the left regular representation of $g$) defined  on $\cG$ by  $\lambda(g)(h) = gh.$  

\medskip

Our aim is to provide an extension of Theorem \ref{aniz} which gives a condition in terms of spectral properties for the mixing time to be minimal.

\medskip

Let $\rho_{\bp}$  and $\sigma_{\bp}$ be the spectral radius and the singular radius of $\cP_{\bp}$ and let $\sigma_{n,\bp}$ be the singular radius of $P_{n,\bp}$ defined in \eqref{eq:specrad}. For integer $t \geq 1$, we define $\sigma_{n,\bp} (t)$ and $\sigma_{n,\bp} (t)$ as in \eqref{eq:defsigt}. From \cite{MR1802431}, the Alon-Boppana lower bound \eqref{lalowertG} is still valid.

\medskip

We wish to focus on sequences of random walks whose spectral gap is uniformly bounded away from one. Hence, we will assume that $\rho_{\bp} < 1$. This is equivalent to assuming that the subgroup $\langle \cA_\bp \rangle$ of $\cG$ generated by $\cA_{\bp}:=\{ g \ : \ p_g>0\}$ is   \textit{non-amenable} (or simply that $\cG$ is non-amenable if one takes as an assumption that $\cA_{\bp}$ generates $\cG$). 
Recall that a group  is said to be \textit{amenable} if it admits a finitely-additive left-invariant probability measure. The equivalent between non-amenability of $\langle \cA_\bp \rangle$ and $\rho_{\bp} < 1$ for $\bp$ 
satisfying \eqref{reversib} was established  by Kesten \cite{MR112053, MR0109367}. In the non-reversible case, see forthcoming Lemma \ref{le:KL}.

\medskip

As before, we can determine an asymptotic lower bound for the mixing time of the random walk with generator $P_{n,\bp}$ (valid for any sequence of group actions) in terms of 
the entropy rate of $\cP_{\bp}$ denoted by $\fh(\bp)$  and defined by Equation \eqref{eq:defhp}. 
In Subsection \ref{subsec:entropiclb} below, we will check that the Avez lower bound $\fh(\bp) \geq - 2 \log \rho_{\bp}$ holds and that  the mixing time of the anisotropic random walk on $G_n$ is asymptotically larger than $(1-o(1))(\log n)/\fh(\bp)$ in the sense  that for any fixed $\veps \in (0,1)$, uniformly in $x\in V_n$, the inequality \eqref{lentroop} holds.

\medskip

For a given probability vector $\bp$ supported by a generator, a natural question is thus the following: {\em Are there spectral conditions for a sequence of actions $(\varphi_n)$ of $\cG$ on $V_n$ with $\# V_n  = n$ to guarantee that the anisotropic random walk on $V_n$ has an asymptotically minimal mixing time in the sense that  $\Tmn(\gep) = (1+ o(1)) ( \log n )/ \fh(\bp)  $ for any fixed $\veps \in (0,1)$?}

\noindent Our answer to this question is based on the two following notions of group algebra.

\begin{definition}[RD property]
A finitely generated group $\cG$ has the RD property (for Markov operators) if for every finitely supported probability vector ${\bf p}$   
the singular radius $\sigma_{\bf p} = \| \cP_{\bp} \|_{2 \to 2}$ of $\cP_{\bp}$  is well controlled by the $\ell^2$-norm of ${\bf p}$ in the following sense :  For any finite symmetric generating set $\cA$ of $\cG$, there exists a constant $C = C(\cG,\cA)$ such that 
\begin{equation}\label{eq:RD}
\sigma_{\bf p} \leq C R^{C}\| \bp \|_2,
\end{equation}
where $R$ is the diameter of the support of $\bp$ in the Cayley graph associated with $(\cG,\cA)$.
\end{definition}

We refer to \cite{MR3666050} for an introduction to the RD property. Among non-amenable groups, we note that free groups and hyperbolic groups satisfy the RD property. Observe also that as the distance corresponding to different generating sets are comparable (if $d_{\cA}$ and $d_{\cA'}$ are the graph distance for the Caley graph associated with generators $\cA$ and $\cA'$ respectively we have $d_{\cA'}\le C_{\cA,\cA'} d_{\cA}$ where $C_{\cA,\cA'}=\max_{y\in \cA'} d_{
\cA}(e,y)$) it is sufficient to check \eqref{eq:RD} for a single finite symmetric generating set $\cA$ of $\cG$.

\medskip

Recall that we automatically have $\sigma_{\bf p} \leq \| \bp \|_1$, and hence \eqref{eq:RD} is trivially satisfied when $\| \bp \|_2\ge C R^{-C}\| \bp \|_1.$ Therefore, the condition \eqref{eq:RD}
says something about $\sigma_{\bp}$ for $\bp$ which have large support and whose mass is well spread on that support.
In fact, as we have $\sigma_{\bf p}\ge \| \bp \|_2$,
the property \eqref{eq:RD} for a non-amenable group tells us in particular that $\sigma_{\bf p}$ is close to this trivial lower bound (in the sense that $\sigma_{\bf p}= \| \bp \|_2^{1+o(1)}$) when $\bp$ is reasonably spread-out on the ball of radius $R$ for large values of $R$ (recall that a non-amenable group has exponential growth). 
\medskip

Our second notion is the {\em strong  convergence of operators algebras} which we define here in our specific Markovian setting. It can be thought as an analog of the Ramanujan property for a sequence of group actions on finite sets.
It is a stronger assumption since the Ramanujan property only refers to one particular random walk on the free group and the property below must be  valid for every random walk.

\begin{definition}[Strong convergence] \label{sc}
Let $\cG$ be a finitely generated group, $(V_n)$ a sequence of finite sets and $(\varphi_n)$ a sequence of covering maps of $\cG$ on $V_n$. We say that the sequence of  covering maps $(\varphi_n)$ converges  strongly (on Markov operators) if for every finitely supported probability vector $\bp$, we have 
$$
\lim_{n \to \infty} \sigma_{n,\bp} = \sigma_\bp.
$$
where $\sigma_{\bp}$ is the singular radius of $\cP_{\bp}$ defined in \eqref{eq:defPp} and  $\sigma_{n,\bp}$ is the singular radius of $P_{n,\bp}$ defined in \eqref{eq:defPnp} and \eqref{eq:specrad}. 
\end{definition}
\noindent From \eqref{eq:gelfand}, the strong convergence of $(\varphi_n)$ implies in particular that 
\begin{equation}\label{consequence}
\lim_{t \to \infty} \lim_{n \to \infty} \sigma_{n,\bp} (t)  = \rho_{\bp}.
\end{equation}
We are now ready to state the last result of this introduction.

\begin{theorem}\label{thegeneral}
Let $\cG$ be a finitely generated non-amenable group with the property RD,  $(V_n)$  a sequence of finite sets with $\# V_n = n$ and $(\varphi_n)$ a sequence of actions of $\cG$ on $V_n$ which converges strongly. Then for every  finitely supported probability vector $\bp$ on $\cG$ such that $\rho_{\bp} < 1$, 
the mixing time of the random walk with transition matrix $P_{n,\bp}$  satisfies, for every $\gep\in (0,1)$,
 \begin{equation*}
  \lim_{n\to \infty} \frac{\Tm_{n,\bp}(\gep)}{\log n}=\frac{1}{\fh(\bp)}.
 \end{equation*}
\end{theorem}

The assumption that the group actions converge strongly is a strong assumption. Notably, Theorem \ref{thegeneral} does not imply neither Theorem \ref{th:LP} nor Theorem \ref{aniz}. These two theorems rely on special properties of free groups.  Nevertheless, in some cases, it is possible to relax the assumption that the group actions converge strongly by supposing instead that the strong convergence holds on some vector spaces of  codimension $n^{o(1)}$ (we  discuss this point  further in Remark \ref{rq:flat2}). 

The paper \cite{BC18} provides a source  of examples for Theorem \ref{thegeneral} by establishing that random actions of the free group are strongly convergent. We consider an involution $*$ in $[d]$ with $q_1$ fixed points, and $\cG^{(q_1,q_2)}_{\mathrm{free}}$  the  group generated by $q_1$ free copies of $\bbZ/2\bbZ$ and $q_2=(d-q_1)/2$ free copies of $\bbZ$ with its natural set of generators. We consider permutations $\alpha_{n,i}$, $i\in[d]$ on $[n]$ vertices which are chosen such that: 
\begin{itemize}
\item [(A)] If $i \ne i^*$, $\alpha_{n,i}$ is a uniform permutation on $n$ elements and $\alpha_{n,i^*}=\alpha^{-1}_{n,i}$.
\item [(B)] If $i=i^*$,   $\alpha_{n,i}$ is a uniform involution on $n$ elements without fixed point (the construction is made only for even $n$).
\item[(C)] The permutations are chosen independently for each equivalence class of the involution $*$.
\end{itemize}

 We consider the action of $\cG^{(q_1,q_2)}_{\mathrm{free}}$ on $V_n =[n]$ defined by $\varphi_n( a_i, x)  = \alpha_{n,i} (x)$. Then, in probability, this sequence of actions is strongly convergent. These random actions on the free group are the only known examples of strongly convergent sequences of actions, but it
could indicate that choosing the action at random among all possible choices might yield a strongly convergent sequence also for other choices of non-amenable groups.

\medskip

\subsection*{Minimal mixing time for color covered random walks.} 
Finally, we also consider yet another extension which allows in particular to consider random walks on $n$-lifts of a base graph (not necessarily regular). Since the work of Amit and Linial \cite{MR1883559,MR2216470} and Friedman \cite{MR1978881},  this model has attracted a substantial attention. In this context, we will give the analog of Theorem \ref{thegeneral}.  To avoid any confusion on notation, we postpone the treatment of this model to Section \ref{sec:lifted}.

\subsection*{Organization of the paper}

 In Section \ref{prel} we provide a short proof for the entropic lower bound \eqref{lentroop} only stetched in this introduction,
 and provide a general result (Proposition \ref{stoop}) which allows to estimate the mixing time of a Markov chain in terms of the distribution of a stopping time at which the chain is close to equilibrium.

In Section \ref{revisited}, we provide a simple proof of Theorem \ref{th:LP}/\ref{goodold}, proving and using a relation between the
$k$-non-backtracking random walk on trees and Chebychev polynomials.

 In Section \ref{sec:thegeneral}, we prove Theorem \ref{sec:thegeneral} concerning cutoff in the more general setup under the assumption of Strong Convergence (Definition \ref{sc}).
 
 In Section \ref{sec:aniz}, we prove Theorem \ref{aniz} concerning anisotropic walks, by combining the ideas of Section \ref{sec:thegeneral} with an analysis of the resolvent of the anisotropic random walk on $\cT_d$.

Finally, in Section \ref{sec:lifted} we deal with the model of color covered random walks.

\subsection*{Notation}
 If $V$ is a countable set and $M$ is a bounded operator in $\ell^2 (V)$, we use the matrix notation $M(x,y) = \langle \ind_x , M \ind_y\rangle$ for $x,y \in V$ where $\ind_x $ denotes the indicator function of $x$. The integer part of real number $t$ is denoted by $\lfloor t \rfloor $.

\subsection*{Acknowledgment} 
The authors are grateful for the   detailed comments of anonymous referees on a first draft of this work which lead to substancial improvement. 
CB warmly thanks Nalini Anantharaman and Justin Salez for enlightening discussions on anisotropic random walks. HL acknowledges support from a Productivity Grant from CNPq and by a JCNE grant from FAPERj.  CB is supported by French ANR grant ANR-16-CE40-0024-01.

\section{Preliminaries}\label{prel}

\subsection{The entropic time lower bound}
\label{subsec:entropiclb}

For the sake of completeness we provide a complete proof of the entropic lower bound \eqref{lentroop}  which is only sketched in the introduction.
The result is stated in the more general setup of Theorem \ref{thegeneral}.
Recall that $\cG$ is a finitely generated group, 
$(V_n)_{n\ge 0}$ a sequence of finite sets with $\# V_n=n$, $(\varphi_n)_{n\ge 0}$ a sequence of actions of $\cG$ on $V_n$,
that $P_{n,\bp}$, $\cP_{\bp}$ denote the transition matrices on $V_n$ and $\cG$ respectively defined by \eqref{eq:defPnp} and\eqref{eq:defPp}, 
and that $\fh(\bp)$ is the entropy rate associated with $\cP_{\bp}$.

\begin{proposition}\label{lb}
Let  $\bp$ be a finitely supported probability vector on $\cG$ such that $\fh(\bp) > 0$.  Given any sequence $(V_n)$, $(\varphi_n)_{n\ge 0}$ as above, 
the mixing time associated with the random walk on $V_n$ with transition $P_{n,\bp}$ satisfies for any $\gep\in(0,1)$
\begin{equation*}
 \liminf_{n} \min_{x \in V_n} \frac{\Tm_{n,\bp}(x,1-\gep)}{\log n}\ge \frac 1 {\fh(\bp)}.
\end{equation*}
\end{proposition} \noindent
We consider $\Tm_{n,\bp}(x,1-\gep)$ (rather than $\Tm_{n,\bp}(x,\gep)$) when lower bounds on the mixing time are concerned so that for both the upper and the lower bound, it is sufficient to consider small values of $\gep$.
\begin{proof}
Let $(\cX_t)$ denote the random walk on $\cG$ starting from the unit $e$ and with transition $\cP_{\bp}$. Its distribution is denoted by $\bbP$. 
The result is an almost direct consequence of the
Shannon-McMillan-Breiman Theorem \cite[Theorem 2.1]{MR704539},
which states that
$\log \cP_{\bp}(e,\cX_t)$ concentrates around its mean; see \eqref{eq:SMMB}. In particular 
given $\gep,\delta>0$ we have for all $t$ sufficiently large 
$$\bbP\big[   \log \cP^t_{\bp}(e,\cX_t) < -(1+\delta) \fh(\bp) t \big]\le \gep/2.$$
Thus if one sets 
$$\mathcal V_{\delta}(t):=\left\{ g\in \cG \,  : \, \cP_{\bp}^t(e,g)\ge e^{-(1+\delta) \fh(\bp)  t} \right\}$$
we have
\begin{equation*}
 |\mathcal V_{\delta}(t)|\le e^{(1+\delta) \fh(\bp)  t} \quad \text{ and} \quad  \bbP[ \cX_t \notin \mathcal V_{\delta}(t) ]\le \gep/2.
\end{equation*}
Now given $x\in V_n$ arbitrary, we consider $X_t:= \varphi_n(\cX_t,x)$, which is a random walk with transition
matrix $P_{n,\bp}$ started at $x$, and let  
$V_{\delta}(t):=\left\{  \varphi_n(g,x) \ : \ g \in \mathcal V_{\delta}(t)\right\},$ the image of $\mathcal V_{\delta}(t)$ by the action.
We have, for all $t$ sufficiently large 
$$\pi_n(V_{\delta}(t))=\frac{|V_{\delta}(t)|}{n}\le \frac{|\mathcal V_{\delta}(t)|}{n}\le  \frac{e^{(1+\delta) \fh(\bp)  t} }{n},$$
and $P^t_{n,\bp}(x, V_{\delta}(t) ) = \dP ( X_t \in V_\delta (t) )  \ge 1- \gep/2$.
Thus we have
\begin{equation*}
 \|P^t_{n,\bp}(x,\cdot)-\pi_n\|_{\TV}\ge 1-\gep/2-  e^{(1+\delta)\fh(\bp)  t} / n.
\end{equation*}
Considering $t= \lfloor  \log (n\gep/2) / ( (1+\delta) \fh(\bp)) \rfloor$, we conclude that for any arbitrarily small and $\gep,\delta>0$, we have for $n$ sufficiently large
$$\Tm_{n,\bp}(x,1-\gep)\ge \lfloor  \log (n\gep/2) / ( (1+\delta) \fh(\bp)) \rfloor\ge (1-\delta)\frac {(\log n)} {\fh(\bp)} .$$ 
It concludes the proof. \end{proof}

The next lemma is the classical Avez lower bound adapted to our definition of spectral radius. It implies notably that if $\rho_\bp < 1$ then $\fh(\bp) > 0$. 
\begin{lemma}\label{le:hrho}
If $\bp$ is a finitely supported probability vector on $\cG$ we have $\fh(\bp) \geq - 2 \log \rho_\bp$.
\end{lemma}
\begin{proof}
We may assume $\rho_{\bp} < 1$. Let $h \geq 0$ be such that $\rho_{\bp} < e^{-h}$. From \eqref{eq:gelfand}, there exists an integer $t_0 \geq 1$ such that for $\| \cP_{\bp} ^{t}  \|^2  = \| \cP_{\bp} ^{t} {\cP_{\bp} ^{t}}^*  \| \leq e^{-2ht}$ for all $t \geq t_0$. In particular, 
$$
 \sum_{g \in \cG} \PAR{\cP_{\bp} ^{t}(e,g)}^2   =  \PAR{\cP_{\bp} ^{t} (\cP^*_{\bp})^{t} } (e,e)  \leq e^{-2 ht}.
$$
From Jensen inequality, we deduce that 
$$
\sum_{g \in \cG} \cP_\bp^{t} (e,g) \log (\cP_{\bp} ^t (e,g) ) \leq   \log \left( \sum_{g \in \cG} \PAR{\cP_\bp^{t} (e,g)} ^2 \right) \leq  - 2ht.  
$$
It follows that $\fh(\bp) \geq 2h$.
\end{proof}
 
\noindent
We conclude this subsection with a corollary of Kesten's criterion for non-amenability applicable to non-reversible walks.
\begin{lemma}\label{le:KL}
Let $\bp$ be a finitely supported probability vector on $\cG$ and $\cA_\bp = \{ g : p_g >0\}$. The subgroup $\langle \cA_ \bp \rangle$ generated by $\cA_\bp$ is non-amenable if and only if $\rho_\bp < 1$.  
\end{lemma}
\begin{proof}
It is convenient to introduce the lazy random walk, $\cL_\bp = (\cI + \cP_\bp)/2= \cP_{\delta_e / 2+\bp/2}$ where $\cI$ is the identity operator. Assume that  $\langle \cA_\bp \rangle$ is non-amenable. Then
$\cL_\bp \cL_\bp ^* $ is of the form $\cP_{\bp'}$ for some $\bp'$ which satisfies \eqref{reversib} and $\langle \cA_{\bp} \rangle = \langle \cA_{\bp'} \rangle$. Kesten's result then implies that the singular radius of $\cL_\bp$ is smaller than $1-\delta$ for some $\delta >0$. The operator norm of operators of the form $\cP_{\bp'}$ with $\bp'$ a probability vector can be obtained by taking the supremum over non-negative functions. We can thus compare the norm  of  $\cL_\bp^{2t} = ((\cI+\cP_{\bp})/2)^{2t}$ with that of any term appearing in its binomial expansion. In particular
we have,
\begin{equation*}
 \|\cP^t_{\bp}  \|_{2 \to 2}  \le 2^{2t}\binom{2t}{t}^{-1} \|\cL_\bp^{2t}\|_{2\to 2} \le 2 \sqrt t \|\cL_\bp \|^{2t}_{2\to 2}   \leq 2 \sqrt t (1 - \delta)^{2t},
\end{equation*}
(where we have used that $\binom{2t}{t} \geq 2^{2t} /(2 \sqrt t)$.
From \eqref{eq:gelfand}, it follows that $\rho_\bp \leq (1-\delta)^2 < 1$.
Conversely, if $\rho_{\bp} < 1$ then, from \eqref{eq:gelfand},  the singular radius of $\cL_\bp^t$ is strictly smaller than $1$ for some $t$. From the definition, it follows that the spectral radius of $\cP_{\bp''} := \cL_\bp^t (\cL_\bp ^t )^*$ is smaller than one, which, by Kesten's reciprocal implies the non-amenability of $\langle \mathcal A_{\bp''} \rangle=\langle  \cA_\bp \rangle$.
\end{proof}

\subsection{Mixing time from stopping time}
\label{subsec:mixstop}
We present here a result derived from 
\cite{MR3650406}, which allows to estimate the distance from equilibrium using arbitrary stopping times.
In this subsection, $(X_t)$ is an arbitrary Markov chain on a finite set $V$ with transition matrix $P$ and for $x\in V$, $\bbP_x$ denotes the distribution of this process with initial condition $X_0 = x$.

\medskip

A classical way to obtain mixing time upper-bounds is via the use of strong stationary times (see \cite[Chapter 6]{MR3726904}). A strong stationary time 
is defined as a stopping time $T$ for the chain $X$ for which $X_T$ is at equilibrium and $X_T$ and $T$ are independent.
The standard bound \cite[Lemma 6.11]{MR3726904} says that if $T$ is a strong stationary time for \eqref{eqdist} then (the bound is in fact proved for the separation distance which is larger)
\begin{equation*}
 \| P^t(x,\cdot)-\pi \|_{\TV}\le \bbP_x[T>t].
\end{equation*}

A careful inspection of the proof in \cite{MR3726904} reveals that one can allow $X_T$ to admit another distribution provided an adequate error term is added. 
However the assumption that $X_T$ and $T$ are independent is crucial in the mechanism of proof.
However using recent techniques developed in \cite{MR3765366,MR3650406} to compare mixing times with hitting times, we can by-pass this independence assumption 
if the chain is reversible and
if the mixing time is much larger than the relaxation time, at the cost of a second error term. We will present a variant of this argument for general finite Markov chains (which in particular does not require reversibility).

\noindent We say that a filtration $(\cF_t)$, is adapted to $(X_t)$ if for any $t\ge 0$, $x, y\in V$. 
\begin{equation}
 \bbE[\ind_{\{X_{t}=x, X_{t+1}=y\}} \ | \ \cF_t]= P(x,y)\ind_{\{X_t=x\}},
\end{equation}
(in particular this implies that $X_t$ is $\cF_t$ measurable).
We denote by $\ell^2(\pi)$ the vector space $\dR^V$ endowed with the scalar product, $$\langle f , g \rangle_{\pi}= \sum_{x \in V} \pi(x) f(x) g(x).$$ Let us recall the definition {\em singular radius} given in \eqref{eq:defsigt} for a finite Markov chain $P$ with invariant measure $\pi$ 
$$
\sigma =  \| P_{|\ind^\perp}\|_{\ell^2(\pi)\to \ell^2(\pi)} = \sup \BRA{ \frac{\| P_{|\ind^\perp} f \|_{\ell^2 (\pi)}}{ \|  f \|_{\ell^2 (\pi)} }  : f \ne 0 }.
$$
For any integer $t \geq 1$, we define the {\em $t$-th singular radius} as 
$$
\sigma(t) = \| (P^t)_{|\ind^\perp}\|^{1/t}_{\ell^2(\pi)\to \ell^2(\pi)}.
$$
Note that $\sigma(t) \leq \sigma$.  
Moreover, in our context, the Poincaré inequality is the claim that for any vector $f \in \dR^V$, with $\pi(f) =  \langle \ind,  f \rangle_{\pi}$,
\begin{equation}\label{eq:poincare}
\| P^t f - \pi(f) \ind \|_{\ell^2(\pi)}  = \| (P^t )_{| \ind^\perp} f \|_{\ell^2(\pi)} \leq \sigma(t)^t  \| f \|_{\ell^2(\pi)}.
\end{equation}
This follows immediately from the definition of $t$-th singular radius.  
We control distance to equilibrium with the help of stopping time with the following result (in the present paper, the inequality \eqref{en3} is sufficient for all purposes, but since
the result might have other applications, we also include a reversible version which is significantly better when $\varrho$ is close to $1$).

\begin{proposition}\label{stoop}
 Let $(X_t)$ be a finite irreducible Markov chain with transition matrix $P$, equilibrium measure $\pi$ and 
 with $t$-th singular radius  $\sigma(t)$.
 If $T$ is a stopping time with respect to a filtration adapted to $X$ and $\bbP_x(X_T\in \cdot)=\nu$, then we have for any positive integers $t$ and $s$:
  \begin{equation}\label{en3}
  \| P^{t+s}(x,\cdot) -\pi\|_{\TV} \le \| \nu -\pi\|_{\TV}+ \bbP_x[T>t]+ 2 (1 - \sigma(s))^{-1/3} \sigma(s)^{2s/3}.
 \end{equation}
 Moreover if $(X_t)$ is reversible and $\varrho = \sigma(1)$ denotes the spectral radius of $P$, we have 
 \begin{equation}\label{en3bis}
  \| P^{t+s}(x,\cdot) -\pi\|_{\TV} \le \| \nu -\pi\|_{\TV}+ \bbP_x[T>t]+ 3 \varrho^{2s/3}.
 \end{equation}
\end{proposition}

\begin{proof}
In the reversible case, the main ingredient of our proof is \cite[Corollary 2.4]{MR3650406}, which we  reformulate as follows.  Given a set $A \subset V$, $s\ge 0$ and $\gep>0$ we set 
$$U(A,s,\gep):=\{ y \in V \ : \ \exists t \ge s,\ | P^t(y,A)- \pi(A) |> \gep \}.$$
Then we have 
$$\pi(U(A,s,\gep))\le 2 \gep^{-2}\varrho^{2s}.$$
From the definition of total variation distance we deduce
\begin{equation}\label{eq:nuU}\nu(U(A,s,\gep))\le 2 \gep^{-2}\varrho^{2s}+\| \nu -\pi\|_{\TV}.
\end{equation}
For every $x$, $t$ and $s$, using the triangle inequality and the fact that $X_T\sim
\nu$ we obtain (using the short-hand notation $U$ for $U(A,s,\gep)$) that
\begin{equation*}
 P^{t+s}(x,A)-\pi(A)\le 
  \sum_{i=0}^t \sum_{y\notin U} \bbP_x(T=i  \ ; X_T=y)(P^{s+t-i}(y,A)-\pi(A))
  + \bbP_x[T>t]+\nu( U).
 \end{equation*}
From the definition of $U(A,t,\gep)$ the double sum above is smaller than $\gep$. Thus, from \eqref{eq:nuU}, we obtain (maximizing over $A$)
\begin{equation*}
 \| P^{t+s}(x,\cdot) -\pi\|_{\TV} \le \bbP_x[T>t]+\| \nu -\pi\|_{\TV}+ 2 \gep^{-2}\varrho^{2s}+\gep,
\end{equation*}
and one can conclude by choosing $\gep=\varrho^{2s/3}$.
In the general case, we define for integer $s \geq 0$
$$
V(A,s,\gep):=\{ y \in V \ :  | P^s(y,A)- \pi(A) |> \gep \} \quad \hbox{ and } \quad U(A,s,\gep) = \bigcup_{t= s}^{\infty} V(A,k,\gep).$$
In particular, we recover the above definition for $U(A,s,\gep)$. 
From Markov inequality and \eqref{eq:poincare}, for any integer $t \geq 0$, we have 
$$
\pi(V(A,t,\gep)) \leq \veps^{-2} \| P^t \ind_A - \pi(A) \ind \|^2_{\ell^2(\pi)}  \leq \veps^{-2} \sigma(t)^{2t} \pi(A).
$$
Since $\sigma_t$ is non-increasing, we thus have for any $s$ and $A$,
$$
\pi ( U(A,s,\gep))  \leq \sum_{k=s}^{\infty} \pi(V(A,k,\gep)) \leq \veps^{-2} (1- \sigma(s)^2)^{-1} \sigma(s)^{2s}.
$$
We deduce a slightly modified form of Equation \eqref{eq:nuU}
$$
\nu(U(A,s,\gep))\le  \veps^{-2} (1- \sigma(s)^2)^{-1} \sigma(s)^{2s} +\| \nu -\pi\|_{\TV}.
$$
We may thus reproduce the same argument. 
 \end{proof}

\section{Simple random walks on Ramanujan graphs revisited}\label{revisited}

\subsection{Sketch of proof of  Theorem \ref{th:LP} and Theorem \ref{goodold}}\label{sketch}

In order to prove Theorem \ref{th:LP} and Theorem \ref{goodold}, we apply 
Proposition \ref{stoop} for a stopping time defined using a coupling between the random walk on $G_n$ and that on $\cT_d$, the infinite $d$-regular tree. This coupling is defined thanks to a covering map from $\cT_d$ to $G_n$.

We denote by $e$ the root vertex of $\cT_d$. Let $\cX$ be the simple random walk on $\cT_d$ starting from $e$.  Given $x\in V_n$, we fix a local graph homeomorphism $\varphi: \cT_d \to G_n$ 
(each vertex $v$ in $\cT_d$ has its $d$ neighbors mapped to the $d$ neighbors of $\varphi(v)$ in $G_n$) such that $\varphi(e)=x$. We may construct the simple random walk on $G_n$ by setting $X_t:= \varphi(\cX_t)$.
For a well chosen integer $k \geq 1$, we define the stopping time $\tau$ as 
\begin{equation}\label{eq:deftau0}
\tau = \inf \{ t \geq 0 : \mathrm{Dist}(\cX_t,e) = k \}, 
\end{equation}
where $\mathrm{Dist}(v,e)$ is the distance of the vertex $v$ in $\cT_d$ to the root $e$. With $k = \frac{\log n}{\log (d-1)}(1+o(1))$, we show that at the time $\tau$, $X_{\tau} =  \varphi(\cX_\tau)$ is close to equilibrium.
More precisely, we use that the distribution of $X_{\tau}$ can be expressed as an explicit polynomial  of the transition matrix $P_n$ (cf. Lemma \ref{polyns}), and thus its $\ell_2$-norm can be controlled in terms of the spectral radius of $P_n$ (cf. Lemma \ref{spectrous}). This spectral bound turns out to be optimal.

Then the proof is concluded easily by using Proposition  \ref{stoop} and the fact that the detailed behavior of $\tau$, which is a hitting time for a biased random walk, is known. It is worth mentioning that this reasoning leads to  a more quantitative result in Proposition \ref{kant} below (which can also be obtained using methods from \cite{MR3558308}). We note also that the variables $X_{\tau}$ and $\tau$ are independent and Proposition  \ref{stoop} in its full strength is not really needed here.

\subsection{Non-backtracking walks and Chebyshev polynomials} \label{nonbac}

In this subsection, let us consider $G=(V,E)$ an arbitrary simple $d$-regular graph.
Given $k\ge 1$ integer, we let 
$$W_k:= \{ (x_{i})_{i=0}^k \in V^{k+1} \ : \ \forall i\in [k], \ \{x_{i-1},x_i\}\in E\}$$ denote the set of paths of length $k$ in $G$.
Given $x,y\in V$, we define the set of non-backtracking paths of length $k$ from $x$ to $y$ as (with the convention that $[0]$ is the empty set)
\begin{equation*}
\NB_k(x,y):=\{ {\bf x}\in W_k \ : \ x_0=x, \ x_k=y \ , \  \forall i\in [k-1], x_{i-1}\ne x_{i+1}\},
\end{equation*}
and $\NB_k(x):=\bigcup_{y\in V} \NB_k(x,y)$.
We define the non-backtracking operator of length $k$ on $G$ to be the following stochastic matrix on $V\times V$,
\begin{equation*}
 Q_k(x,y):= \frac{\#\NB_k(x,y)}{\#\NB_k(x)}=\frac{\#\NB_k(x,y)}{d(d-1)^{k-1}}.
\end{equation*}

We let $P$ denote the transition matrix for the simple random walk on $G$ (i.e. $P=Q_1$).
The following well known result (see e.g. \cite{MR2348845,MR2377835} and \cite{MR0338272} for an early reference) will help us to control the spectral radius of $Q_k$ in terms of that of $P$.
\begin{lemma}\label{polyns}
 For every integer $k$, there exists a polynomial $p_k$ such that $Q_k=p_k(P)$ for every simple $d$-regular graph $G$.
 More precisely we have
 $$
p_k  (x) = \frac{1}{d  (d-1)^{k/2}} \PAR{   (d-1) U_k \PAR{\frac{x }{  \rho}  }  - U_{k-2} \PAR{\frac{x }{  \rho}  } },
$$
where $\rho: = ( 2 \sqrt {d-1} ) / d $ and $(U_k)$, $k \geq -1$, are the Chebyshev polynomials of the second type defined rescursively by
$$U_{-1} = 0, \quad U_0 = 1, \quad  \text{ and } \quad  U_{k+1}(x) = 2 x U_k (x)- U_{k-1}(x).$$ 
 \end{lemma}
\begin{proof}[Proof (sketch)]
For a more detailed proof, we refer to the above mentionned references \cite{MR2348845,MR0338272,MR2377835}. It is sufficient to check that the identity $Q_k=p_k(P)$ is valid on the $d$-regular infinite tree $\cT_d$ (it is the universal covering of $G$ and the preimage by a covering map of the non-bactracking paths on $G$ are  the non-backtracking paths on $\cT_d$). 
 We set $\bar Q_k:= d(d-1)^{k-1} Q_k$.
Using that $\bar Q_k(x,y)=\ind_{\{\mathrm{Dist}(x,y)=k\}}$ it is simple to check that
 \begin{equation}\label{eq:recQbar}
  \bar Q_{k+1}=   d P \bar Q_k - \bar Q_{k-1}. 
 \end{equation}
The result follows then by induction on $k$. We find $p_1 = x$, $p_2 = x^2   d / (d-1) - 1 / (d-1) $
and, from \eqref{eq:recQbar},
$$p_{k+1} = \frac {d }{ d-1}  x p_k  - \frac{ 1 }{ d-1} p_{k-1}.$$
It is then immediate to check that this recursion coincides with the recursion satisfied by the polynomials $d^{-1}  (d-1)^{-k/2} \PAR{   (d-1) U_k \PAR{ x  /  \rho }  - U_{k-2} \PAR{\ x  / \rho  } }$.  The conclusion follows. \end{proof}

\noindent The polynomials $(p_k)$ are called the {\em Geronimus polynomials} (in reference to \cite{MR1502972}) or the {\em non-backtracking polynomials}.

\subsection{Proof of Theorem \ref{th:LP}}

Recall that $\rho_n$ denotes the spectral radius for $P_n$ restricted to non-constant function. We let 
$$\eta_n:= \max\left(0,\frac{d \rho_n }{2\sqrt{d-1}}-1\right)$$ quantify by how much $G_n$ is far from being a Ramanujan graph. Theorem \ref{th:LP} is a consequence of this more quantitative statement.
\begin{proposition}\label{kant}
Let $(G_n)$ be a sequence of $d$-regular graphs on $n$ vertices such that $\lim_{n\to \infty}\eta_n=0$.
There exists a constant $C$ and a sequence $\delta_n$ tending to zero such that for all $\gep\in(0,1)$ for all $n$ sufficiently large (depending on $\gep$)
\begin{equation}\label{lupper}
 \Tmn(\gep)\le \left( \frac{ d }{(d-2) \log (d-1)}+ C\sqrt{\eta_n} \right) \log n+ (\Phi(\gep)+\delta_n) \sqrt{\log n},
\end{equation}
where, if $Z$ is a standard normal distribution, the function $\Phi(\cdot)$ is defined as the inverse of 
$$s\mapsto \dP\left[ Z \ge \frac{(d-2)^{3/2}}{2\sqrt{d(d-1)}}s\right].$$
In particular if $\lim_{n\to \infty}\eta_n \log n=0$,
then 
\begin{equation}\label{profyle}
\Tmn(\gep)= \frac{ d }{(d-2) \log (d-1)} \log n + \Phi(\gep)\sqrt{\log n}+o(\sqrt{\log n}).
\end{equation}
\end{proposition}

Note that the upper-bound in \eqref{profyle} is an immediate consequence of \eqref{lupper}, while the lower bound (displayed in \cite[Fact 2.1]{MR3558308}) - which is valid for any $d$-regular graph - follows from the argument sketched in the introduction. We note also that it follows from \cite{bordenaveCAT} that, if $G_n$ is a uniform random $d$-regular graph on $n$ vertices then $\eta_n \sqrt{\log n}$ converges to $0$ in probability. Hence, we recover the main result of \cite{MR2667423} from Proposition \ref{kant}.
The remainder of this subsection is devoted to the proof of Proposition \ref{kant}. The proof includes a  technical lemma whose proof is postponed to the end of the section.

\begin{proof}[Proof of Proposition \ref{kant}]
We apply the content of the previous subsection to our problem. Let $x$ be in $V_n$ and $\varphi : \cT_d \to G_n$ be as in Section \ref{sketch} be a local graph homeomorphism such that $\varphi(e) = x$, where $e$ is the root of $\cT_d$.  Let $\cX_t$ be the simple random walk on $\cT_d$ started at the root vertex $e$.  Then $X_t := \varphi(X_t)$ is a simple random walk on $G_n$  starting from $x$.  For an integer $k$ to be chosen later on, let $\tau $ be defined as in \eqref{eq:deftau0}. Since non-backtracking paths in a tree are geodesic paths, it is immediate to see that the distribution of $X_{\tau}$ is given by $Q_{k,n}(x,\cdot)$ where $Q_{k,n}(x,\cdot)$ is the non-backtracking operator on $G_n$.
Hence in particular the standard $\ell_2$ upper-bound on total variation distance \eqref{standspec} applied for $t=1$ yields 
\begin{equation}\label{rouf}
 \|Q_{k,n}(x,\cdot)- \pi_n\|_{\TV}\le r_{k,n}\sqrt{n}.
\end{equation}
where, using Lemma \ref{polyns}
 $$ r_{k,n}:= \max_{\gl \in \Sp(Q_{k,n})\setminus 1} |\gl|= \max_{\gl \in \Sp(P_n)\setminus 1} p_k(\gl).$$
Hence if one sets 
$$k = k_n:= \min\left\{ k  \ : \  r_{k,n}\le \frac{1}{\sqrt{n} \log n} \right\},$$
we deduce from \eqref{rouf} that
$
 \|Q_{k_n,n}(x,\cdot)- \pi_n\|_{\TV} \leq (\log n)^{-1}
$. 
We now apply Proposition \ref{stoop} for $T=\tau$. We obtain that, provided that $\rho_n \le 1-\delta$ (which is true for all $n$ large enough if $\eta_n\to 0$ for e.g. $\delta=1/20$), for all $t\geq 0$,
\begin{equation}\label{lappli}
 d_n(x,t+s)\le  \|Q_{k_n,n}(x,\cdot)- \pi_n\|_{\TV}+ \bbP[\tau \ge t ]+ 3(1-\delta)^{2s/3}.
\end{equation}
The last term can be made smaller than $(\log n)^{-1}$ for all $n$ large enough by choosing $s = s_n:= (\log \log n)^2$. Hence, setting 
$$ t_n(\gep):= \inf\left\{ t \ : \  \bbP[\tau > t ]\le \gep -2  (\log n)^{-1}\right\},$$ 
we obtain
$$\Tmn(\gep)\le t_n(\gep)+ s_n.$$
Now, the central limit theorem for the biased random walk on the line implies that 
$$
\frac{\mathrm{Dist} ( \cX_t , o ) - t( (d-2)/d) }{ 2 \sqrt{d-1} / d } 
$$
converges weakly to a standard normal distribution. We may thus easily estimate $t_n$ as a function of $k_n$.  Hence, the only missing part is an estimate for $k_n$.

\begin{lemma}\label{spectrous}
For any integer $d \geq 3$, there exists a constant $C$ such that for all $n$ sufficiently large we have
\begin{equation*}
 k_{n}\le \begin{cases}
              \frac{\log n}{\log (d-1)} + C\left( \log\log n \right), \quad  &\text{ if } \eta \le (\log n)^{-2}(\log \log n)^2, \\
              \frac{\log n}{\log (d-1)}  (1+C\sqrt{\eta}) \quad \quad &\text{ if } \eta \ge (\log n)^{-2}(\log \log n)^2.
             \end{cases}
\end{equation*}
\end{lemma}
The above estimates combined with the use of the Central Limit Theorem (details are left to the reader) imply that 
 $$t_n(\gep)\le   \left( \frac{ d }{(d-2) \log (d-1)}+ C\sqrt{\eta_n} \right) \log n+\left(\Phi(\gep)+\delta_n\right) \sqrt{\log n} .$$
This concludes the proof of Proposition \ref{kant} (provided that Lemma \ref{spectrous} has been established). \end{proof}

\begin{proof}[Proof of Lemma \ref{spectrous}]
We use the following classic identities, valid for all $\theta \in \bbR\backslash\{0\}$ and $k \in \dN$,
\begin{equation}\label{quant}
U_k (\cosh \theta) = \frac{\sinh ( (k+1) \theta )}{ \sinh \theta} \quad \hbox{ and } \quad  U_k (\cos \theta) = \frac{\sin ( (k+1) \theta )}{ \sin \theta}. 
\end{equation}
We note that $U_k$ is either an even or an odd function (depending on the parity of $k$). We thus have for any $\gl$,
$$|p_k(\gl)| \le \frac{1}{d(d-1)^{k/2}}[(d-1)|U_k( |\gl|/\rho)|+|U_{k-2}(\gl/\rho)|] $$
Using the fact (it can be checked using \eqref{quant}) that $|U_k(x)|\le k+1$ on $[0,1]$ and $U_k(x)$ is increasing on $[1,\infty)$ we obtain that 
\begin{equation}\label{loop}
\max_{\gl \in \Sp(P_n)\setminus 1} |p_k(\gl)|
\le \frac{1}{d(d-1)^{k/2}}[(d-1)U_k(1+\eta_n)+U_{k-2}(1+\eta_n)].
\end{equation}
and hence $r_{k,n}\le (d-1)^{-k/2}U_k(1+\eta_n)$.
Using  the identity \eqref{quant} we obtain that there exists a constant $C$ such that
\begin{equation*}\label{laborn}
 r_{k,n} \le \frac{C}{(d-1)^{k/2}} \min(\eta^{-1/2},k)e^{C k\sqrt{\eta}}.
\end{equation*}
This is sufficient to obtain the desired estimate on $k_n$.\end{proof}

\subsection{Proof of  Theorem \ref{goodold}}

Let $\eta >0$. To prove Theorem \ref{goodold}, we use \eqref{lappli} with $k_n$ replaced by $k'_n= \frac{\log n}{\log (d-1)}(1+\eta/2)$.
By the law of large numbers $\tau = \tau(k)$ is asymptotically equivalent to $\frac{k d}{d-2}$ when $k$ goes to infinity. Hence,  to prove Theorem \ref{goodold}, it is sufficient to show that, there exists $c >0$ such that for all $n$ large enough, we have 
$\| Q_{k'_n,n}(x,\cdot) - \pi_n \|_{\TV} \leq n^{-c\eta}$ for at least $n^{1 - 2c \eta}$ vertices $x$ in $V_n$. It is thus sufficient  to show that for all $n$ sufficiently large, 
\begin{equation}\label{objectif}
 \sum_{x\in V_n}\| Q_{k'_n,n}(x,\cdot) - \pi_n \|_{\TV}\leq  n^{1-3c\eta}.
\end{equation}
To take into account the information we have about the multiplicity of eigenvalues, we must be more precise than \eqref{rouf} in our decomposition.
For $\gl \in \Sp(P_n)\setminus \{1 \}$ we let $\alpha_\gl(x)$ be the square norm of the projection
of the vector $\delta_x$ onto $E^{\gl}_n$, the eigenspace of $P_n$ corresponding to $\gl$, that is 
$$\alpha_{\gl}(x):= \max_{f\in E^{\gl}_n } \frac{  f(x)^2}{ \sum_{y\in V_n} f(y)^2}.$$
From the spectral theorem,  we have $\sum_{x\in V_n} \alpha_{\gl}(x)= \mathrm{dim}( E^{\gl}_n)$.
Using the Cauchy-Schwartz inequality and the decomposition on the eigenspaces of $P_n$ we obtain
\begin{equation*}\label{ptitun}
 \left( 2\| Q_{k,n}(x,\cdot) - \pi_n \|_{\TV}\right)^2
 \le  n \sum_{y\in V_n}  \left(Q_{k,n}(x,y)-\frac 1 n \right)^2
 =   \sum_{\gl \in \Sp(P_n)\setminus \{1 \}}  n p_{k}(\gl)^2 \alpha_\gl(x).
\end{equation*}
Hence, averaging over $x$, we have 
\begin{equation*}
\frac 1 n \sum_{x\in V_n} \| Q_{k,n}(x,\cdot) - \pi_n \|_{\TV}^2 \le  \sum_{\gl \in \Sp(P_n)\setminus \{1 \}}  p_{k}(\gl)^2 \mathrm{dim}( E^{\gl}_n).
\end{equation*}
Using the fact (recall \eqref{loop}) that $p_k(\gl)\le (d-1)^{-k/2}(k+1)$ when $\gl\le \rho$, we obtain that 
\begin{equation}\label{leptit}
  \sum_{\gl \in \Sp(P_n)\cap [-\rho,\rho]}  p_{k}(\gl)^2 \mathrm{dim}( E^{\gl}_n) \le  (d-1)^{-k} (k+1)^2 n.
\end{equation}
For $\gl\notin [-\rho,\rho]$, as a consequence of  \eqref{quant} we have
\begin{eqnarray*}\nonumber%\label{uniff}
 \limsup_{k\to \infty} \frac{1}{k} \log |p_k(\gl)|&\le&  \lim_{k\to \infty} \frac{1}{k} \log U_k(|\gl|/\rho)- \frac{1}{2}\log (d-1)\\
& =&  \log\left(|\gl|/\rho- \sqrt{\left(\gl/\rho\right)^2-1}\right)- \frac{1}{2}\log (d-1).
\end{eqnarray*}
Hence recalling the definition of $I(u)$ in  \eqref{azymp} and the assumption $\rho_n\le 1-\delta$, we arrive at
\begin{multline}\label{legros}
   \limsup_{n\to \infty} \frac{1}{\log n} \log \sum_{\gl \in \Sp(P_n)\setminus [-\rho,\rho]}  p_{k'_n}(\gl)^2 \mathrm{dim}( E^{\gl}_n)
   \\
   \le \sup_{u\in [\rho,1-\delta]}\left[ (1+\eta/2)\left( \frac{2\log\left(u/\rho- \sqrt{\left(u/\rho\right)^2-1}\right)}{\log (d-1)}- 1\right)+I(u)\right],
\end{multline}
(where we have used the upper semi-continuity of $u \mapsto I(u)$). 
Using the assumption \eqref{supoz}, we obtain that the left-hand side of \eqref{legros} is at most $c_0 \eta$ with 
$$c_0:= \frac 1 2  -\frac{\log\left((1-\delta)/\rho- \sqrt{\left((1-\delta)/\rho\right)^2-1}\right)}{\log (d-1)}>0.$$
Together with \eqref{leptit}, it concludes the proof of \eqref{objectif} with $c = c_0/4$. \qed

\begin{remark}[Variant of Theorem \ref{goodold}]\label{rq:flat0}
If $H$ is a vector space of $\dR^{V_n}$ with $\# V_n = n$, we define the {\em flat-dimension} of $H$ as $\dim_0 (H) =  n  \max_{x \in V_n} \| P_H \ind_x \|^2_2$ where $P_H$ is the orthogonal projection onto $H$. This definition implies $\dim_0 (H) \geq  \dim(H)$, $\dim_0 (\mathrm{span}(\pi_n)) = 1$ and $\dim_0 (\mathrm{span}(\ind_x)) = n $.  If the  graph $G_n$ is a transitive graph and $H$ is the invariant vector space generated by $k$ eigenvalues of  $P_{n}$, then we have  $\dim_0 (H) = \dim(H)$. Now, we define $I_0$ exactly as $I$ in \eqref{azymp} except that we replace in  \eqref{azymp} $\dim(E_n^{\lambda})$ by $\dim_0(E_n^{\lambda})$. The proof of Theorem \ref{goodold} actually proves that \eqref{kutof} holds if $\rho_n < 1 - \delta$ and for all $u > \rho$, \eqref{supoz} holds with $I_0$ instead of $I$.
\end{remark}

\section{Covered random walks: proof of Theorem \ref{thegeneral}}

\label{sec:thegeneral}

\subsection{Notation} In this section, we fix a finitely supported probability vector $\bp$ on $\cG$ and we denote by $(\cX_t)_{t \geq 0}$ the random walk on $\cG$ with transition kernel $\cP_{\bp}$ started at $\cX_0 = e$, the unit of $\cG$. The underlying probability distribution of the process $(\cX_t)_{t \geq 0}$ on $\cG^{\dN}$ will be denoted by $\dP (\cdot)$. Finally, if $\varphi_n$ is the action of $\cG$ on $V_n$ as in Theorem \ref{thegeneral}, given a fixed $x\in V_n$, we set $X_t = \varphi_n(\cX_t,x)$. Then $(X_t)_{t\geq 0}$ is a trajectory of the Markov chain on $V_n$  with transition kernel $P_{n,\bp}$ and initial condition $x$.

\subsection{Proof strategy}

Our strategy shares some similarities with that adopted in the Ramanujan case: we try to build a {\em backbone walk} $(Y_s)_{s \geq 0}$ with $Y_s = X_{\tau_s}$ using stopping times which are defined in terms of the walk performed on the covering (our terms backbone comes from the fact that the complete walk can be recovered from the backbone by adding the missing pieces).
The two important properties that our backbone walk must satisfy are the following :
\begin{itemize}
 \item[(i)] At each step, one jumps more or less uniformly to one of $k$ vertices for a large $k$.
 \item[(ii)] The spectral gap associated with the backbone walk is close to the Alon-Boppana bound.
\end{itemize}

The second property is obtained from our assumptions  that the RD property on $\cG$ holds and that the sequence of actions  converges strongly. 
To obtain a backbone walk that jumps close to uniformly on large sets, we perform an explicit construction based on the Green's operator associated to $\cP_{\bp}$.

\medskip

To conclude, we need to relate the mixing time of the backbone walk to that of the original one. This is done using the tools developed in Subsection \ref{subsec:mixstop} which relate mixing times and hitting times. Indeed hitting times of the backbone walk provide an upper bound for the hitting times of the original walk.

\subsection{Construction of the backbone walk from the Green's operator} \label{skelex}

Given $k$ a large integer, our task is to find a stopping time $\tau$ for the process $(\cX_t)$ starting from $\cX_0 = e$ such that $\cX_{\tau}$ is close to be uniformly distributed on a set of $k$ vertices. We denote by $\cA = \{ a_1, \ldots, a_d \}$ be the symmetric support of $\cP_{\bp}$.  We define $\Gamma$ as the Cayley graph of $\cG$ generated by $\cA$. By construction $(\cX_t)$ is a random walk on $\Gamma$.  We are going to choose our stopping time of the form 
\begin{equation}\label{eq:deftau}
\tau := \inf \BRA{ t \geq 0 : \cX_t \notin \cU },
\end{equation}
where $\cU$ is finite and contains $e$. 
Notably, $\tau$ is almost surely finite and $\cX_{\tau}$ is supported on the set $\partial \cU $ defined by 
$$\partial \cU: = \{ g  \notin \cU : a^{-1}_i g \in \cU \hbox{ for some $i \in [d]$}\},$$ which satisfies $\# \partial \cU\le (d-1)\#\cU$.

\medskip

Now let us specify our choice for $U$. 
We let $\cR_{\bp}= (\cI -  \cP_{\bp} )^{-1}$ be the Green's operator associated with $\cP_{\bp}$. The Green's operator is a well defined bounded operator,  since $\rho_\bp < 1$, $1$ is not in the spectrum of $\cP_{\bp}$. 
 We define $u$ to be the image of the coordinate vector at $e$ by $\cR_{\bp}$. The scalar $u(g)$ corresponds to the expected number of visits at $g $ starting from $e$: 
\begin{equation} \label{defuggreen}
u(g) :=  \cR_{\bp} (e,g) = \sum_{t = 0}^ {\infty} \cP_{\bp}^t (e,g).
\end{equation}
Given $k \geq 1 $ we define the set
 \begin{equation}\label{defuu}
\cU := \BRA{ g \in \cG : u (g) >  \frac 1 k }. 
\end{equation}
In the reversible case,
our set $U$ can be interpreted as the complement of a ball around $e$ in the Green metric on $\cG$ associated to $\cP_{\bp}$. The Green metric is defined for all $g,h \in  \cG$ by $d(g,h) = \log \cR_{\bp} (e,e) - \log \cR_{\bp} (g,h)$ and it is closely related to the entropy of the random walk $(\cX_t)$, see \cite{MR2408585}.
Our backbone walk is the induced walk on the successive exit times from $\cU$.
More precisely, we define  $\tau_0 := 0$, $\tau_1 = \tau$ and, for integer $s \geq 1$, $\tau_{s+1} := \inf \{ t \geq \tau_s : \cX_{t} \cX^{-1}_{\tau_s} \notin \cU \}$. We finally 
set $\cY_s:= \cX_{\tau_s}$. 
We  denote by $\cQ$ the transition kernel associated with the  Markov chain $(\cY_s)$: for   any $g, h \in \cG$, 
\begin{equation}\label{firstdefq}
\cQ ( g,h) = \dP ( \cX_{\tau} = h g^{-1} ). 
\end{equation}
By construction,  we have $\cQ = \cP_{\bq}$  where, for all 
$g \in \cG$,
\begin{equation}\label{eq:defQ}
q_g := \cQ(e,g) =\bbP\left[ \cY_{1}=g \right] = \bbP\left[ \cX_{\tau}=g \right].
\end{equation}
% Note also that from \eqref{eq:symqi}, $U^{-1} = U$ and the vector $\bq = (q_g)_{g \in \cG}$ satisfies the right-hand side of \eqref{eq:symqi}. 
We let $(Y_s)$ be the projection of the walk $(\cY_s)$ onto 
$V_n$, $Y_s:=\varphi_n(\cY_s,x)$, and let $Q_n$ denote the associated transition kernel. This Markov chain is our backbone walk. The following result establishes that $U$ has the desired property.  
\begin{proposition}\label{le:UnifU}
Assume that $\rho_{\bp} < 1$. Then there exists a constant $C$ such that 
 for every integer $k\geq 2$, the set defined by Equation \eqref{defuu} satisfies  $e \in \cU$,
 $\#\cU \le C k\log k$, $\Diam (\cU) \le C\log k$ (where $\Diam$ denotes the diameter for the graph distance in $\Gamma$) and
 such that for $\bq$ defined by \eqref{eq:defQ},  
\begin{equation}\label{eq:tandilt}
\forall g \in \partial U,\quad q_g \le \frac{1}{k}.
\end{equation}
\end{proposition}

\begin{proof}
By definition of the function $u$, we have $e \in \cU$ and for any $g\in \partial U$, 
$
q_g = \bbP\left[ \cX_{\tau}=g \right]\le u(g)\le 1/k,
$
as requested.  We now  check that the cardinality of $\cU$ is controlled by $k\log k$. This is a simple consequence of the assumption that $\rho_{\bp}<1$. We fix $\rho$ such that $\rho_{\bp} < \rho < 1$. Then, from \eqref{eq:gelfand}, there exists $s \geq 1$, such that $\| \cP^t_{\bp} \| \leq \rho^t$ for all $t \geq s$. Hence, there exists a constant $C_0 \geq 1$, such that  $\| \cP^t_{\bp} \| \leq C_0 \rho^t$ for all $t \geq 1$.   Notably, we deduce that for all $g,h \in \cG$,
\begin{equation}\label{eq:Ptinf}
\cP_\bp^t (g,h) \leq  \| \cP_\bp^t \|  \le C_0 \rho^t,
\end{equation}
(see forthcoming Lemma \ref{le:entriesrho} for an improvement of this inequality). Thus, if $\mathrm{Dist}(g,e) $ is the graph distance between $g$ and $e$ in  $\Gamma$,
$$
u(g) = \sum_{t \geq |g|} \cP_\bp^{t} (e,g) \leq C_0 (1-\rho)^{-1}\rho^{\mathrm{Dist}(g,e)}.   
$$
This implies that $U$ is included in the ball $B_r$ of radius 
$r=\lfloor C_1\log k \rfloor $ around the unit $e$. 
For any integer $b \geq 1$, we find
 \begin{eqnarray*}
\frac{\# U}{k} &\le & \sum_{g \in U} u(g)  \leq  \sum_{g \in B_r } \sum_{ t = 1}^\infty \cP_\bp^t (e,g) \\
& \leq & \sum_{t=1}^{ b r}  \sum_{g \in B_r }  \cP_\bp^t (e,g) +  \sum_{t= br +1}^{ \infty}  \sum_{g \in B_r }  \cP_\bp^t (e,g)\\
& \leq &  br + \sum_{t  = br +1} ^\infty (\# B_r )    C_0 \rho^t.
\end{eqnarray*}
We choose $b >0$ such that $(d-1) \rho^b < 1$. Since $\# B_r  \leq d (d-1) ^{r-1}$, we thus find that $\# U / k $ is at most $C_2  \log k$ as requested (with $C_2=2bC_1$). \end{proof}

\subsection{Deducing mixing times from RD property and the strong convergence}\label{sec:deduce}

To compare the original walk with the backbone walk, the first requirement is to control how much time each backbone step requires on average. 
This can be deduced from the definition of the entropy of $\cG$.  Recall the definition of $\tau$ in \eqref{eq:deftau}.

\begin{lemma}\label{le:tau} Assume that $\rho_{\bp} < 1$. 
For any $\veps >0$, there exists $k(\veps) >1$ such that for all integers $k \geq k(\veps)$,
$$
\dE [\tau] \leq ( 1+ \veps) \frac{\log k }{\fh(\bp)}. 
$$
\end{lemma}

\begin{proof} Given $t_1<t_2<\infty$, we decompose the expectation in three contributions ($\tau\le t_1$, $\tau\in (t_1,t_2]$, $\tau>t_2$) and obtain
\begin{equation}\label{eq:3parts}
 \bbE[\tau]\le t_1+ t_2\bbP(\tau>t_1)+ \bbE[\tau \IND_{\{ \tau >t_2\}}].
\end{equation}
We set $$t_1:= ( 1+ \veps/2) \frac{\log k }{\fh(\bp)} \quad  \text{  and } \quad  t_2:= C \log k$$ for some adequate constant $C$, and prove that the second and third term in 
\eqref{eq:3parts} are smaller than $(\gep/4)(\log k /\fh(\bp))$. 
We start by bounding the tail probability of $\tau$. Recall that $\rho_\bp$ is the spectral radius of $\cP_\bp$. We fix $\rho$ such that $\rho_{\bp} < \rho < 1$. From \eqref{eq:Ptinf},
$$
\dP ( \tau > t ) \leq \dP ( \cX_t \in U ) \leq \sum_{g \in U} \cP_\bp^t (e,g) \leq  \# U C_0 \rho^t.
$$
Hence, for any $s>0$,
$$
\dP \PAR{ \tau > \frac{\log( C_0 \# U ) + s }{\log( 1 / \rho)}  } \leq e^{-s}.  
$$ 
By Proposition \ref{le:UnifU}, we deduce for some choice of  constant $C>0$, for any $s>0$ and integer $k \geq 2$, 
\begin{equation*}\label{eq:tailT}
\dP \PAR{ \tau > (C/2) (\log k  + s )  } \leq e^{-s}.  
\end{equation*}
It follows that for any $\veps >0$, for all $k$ large enough 
$$
\dE \left[\tau \IND_{\{ \tau > t_2\} } \right]\leq \frac{1}{\log k} \leq \frac{ \veps \log k }{4 \fh(\bp)}.
 $$

\noindent Now to control the second term, we need to show that 
$\bP(\tau>t_1)\le \veps/(4C\fh(\bp))$.
Set
$$
H = \BRA{ g \in \cG : \cP_\bp^{t_1} (e,g) \leq e^{ - ( 1 -\veps/3)  t_1 \fh(\bp)} },
$$
and arguing as above,
$$
\dP ( \tau > t_1 ) \leq \dP ( \cX_{t_1} \in U  )  \leq \sum_{g \in U \cap H} \cP_\bp^{t_1} (e,g) + \dP ( \cX_t \notin H). 
$$ 
Now, from \eqref{eq:SMMB}, if $k$ is large enough, $\dP ( \cX_t \notin H) \leq \veps/(8C\fh(\bp)) $ and, by Proposition \ref{le:UnifU},
$$
 \sum_{g \in U \cap H} \cP_\bp^{t_1} (e,g)  \leq \# U e^{ -( 1 -\veps/3) t_1 \fh  (\bp)} \leq   (C k \log k ) k^{-(1 + \veps/10)}\le  \veps/(8C\fh(\bp)),
$$
as requested.
\end{proof}

\begin{remark}\label{rq:lbtau}
The above proof actually shows that the conclusion of Lemma \ref{le:tau} is true for any exit time from a set of cardinality $k^{1 + o(1)}$. On the other hand \eqref{eq:SMMB} and the lower bound $u(\cX_t) \geq \cP_\bp^t (e,\cX_t)$ imply easily that $\dE [ \tau ] \geq ( 1 -  \veps) ( \log k  ) / \fh(\bp)$ for all $k$ large enough. Hence, our set $U$ asymptotically maximizes  the mean exit time (among all sets of cardinality $k^{1+ o(1)}$).
\end{remark}

\noindent All ingredients are now gathered to conclude. 
\begin{proof}[Proof of Theorem \ref{thegeneral}]

We fix $\veps \in (0,1)$, $\delta>0$  and $x\in V_n$ arbitrary and prove that for $n$ sufficiently large
$$\Tm_ {n,\bp} (x,\gep)\le   (1+\delta)\log n /\fh(\bp).$$ 
We consider $\tau$ constructed with $U$ from Proposition \ref{le:UnifU} for some large $k$ which we are going to choose depending on $\delta$ but not on $n$, and we set  $m:= \lfloor (1+\delta/4) (\log n ) / \log k \rfloor$.
We use Proposition \ref{stoop} for the walk $X_t:=\varphi_n(\cX_t,x)$ with
$$T = \tau_m, \quad t=t_n=\lfloor (1+\delta)\log n /\fh(\bp)-  \log\log n \rfloor \quad \text{ and } s=s_n=\lfloor \log \log n\rfloor .$$
 We have
  \begin{equation}\label{en3ter}
  \| P^{t_n+s_n}_{n,\bp}(x,\cdot) -\pi\|_{\TV} \le \| Q^m_n -\pi\|_{\TV}+ \bbP[\tau_m>t_n]+ 2 (1 - \sigma_{n,\bp}(s_n))^{-1/3} \sigma_{n,\bp}(s_n)^{2s_n/3}.
 \end{equation}
 We are going to show that for $n$ sufficiently large each of the three terms in the r.h.s.\ are smaller than $\gep/3$.

We start with the third term. To deal with it we prove the following statement  
\begin{equation}\label{lalimsup}
\limsuptwo{n\to \infty}{s \to \infty} \sigma_{n,\bp}(s)\le \rho_{\bp},
\end{equation}
where the limit can be taken over arbitray sequences of $n$ and $s$ which both go to infinity (though it is sufficient for our purpose to know that the $\limsup$ is $<1$).
 If $s = a s_0 + b$ with $a,b,s_1$ non-negative integers, we have $$\| (P_{n,\bp}^s) _{\ind^\perp} \|_{2 \to 2} \leq \| (P_{n,\bp}^{s_0})_{\ind^\perp} \|_{2 \to 2}^{a}\| (P_{n,\bp}^{b})_{\ind^\perp} \|_{2 \to 2}\le \| (P_{n,\bp}^{s_0})_{\ind^\perp} \|_{2 \to 2}^{a}
.$$ Applied to $ a= \lfloor s / s_0 \rfloor$, we deduce that for all $t \geq t_1$,
\begin{equation}\label{eq:sms}
\sigma_{n,\bp} (s) \leq \sigma_{n,\bp}(s_0)^{1 - \frac{s_0-1}{s}}.
\end{equation}
From \eqref{consequence}, we can choose $s_0(\delta)$ and $n_0(\delta)$ such that $\sigma_{n,\bp}(s_0)\le \rho_{\bp}+\delta/2$ and we can deduce from \eqref{eq:sms} that 
$ \sigma_{n,\bp} (s) \leq \rho_{\bp}+\delta$ for $n\ge n_0(\delta)$ and $s\ge s_1 (\delta)$ sufficiently large.
Let us now move the  second and third term in \eqref{en3bis}

The probability $\bbP[\tau_m>t_n]$ is small  as a consequence of the law of large numbers. Indeed choosing $k(\delta)$ sufficiently large Lemma, \ref{le:tau},
guaranties that $t_n\ge (1+\delta/2) m \bbE[\tau]$.
The smallness of $\| Q^m_n -\pi\|_{\TV}$ is obtained using spectral estimates for $Q_n$. Since $\cG$ has the RD property \eqref{eq:RD}, we deduce from Proposition \ref{le:UnifU} that for some constants $C, C'$ (depending on $\bp$)  

\begin{equation*}
\sigma_{\bq} = \| \cP_{\bq} \| \leq C ( \log k)^{C}\left( k^{-2}\# \partial U \right)^{1/2}
\le C' k^{-1/2} (\log k)^{C+1/2}.
\end{equation*}
 Now, the assumption that $(\varphi_n)$ converges strongly applied to $\bq$ implies that for all $n$ large enough (depending on $k$), the singular radius $\sigma_{n,\bq}$  of $Q_n$ satisfies.  
\begin{equation*}\label{zebound}
\sigma_{n,\bq} \leq 2C' k^{-1/2} (\log k)^{C+1/2}.
\end{equation*}
Then, we use the Cauchy-Schwarz inequality and  the usual $\ell^2$-distance bound \eqref{eq:poincare}. We obtain that for any $m\ge 1$,
\begin{equation}\label{spectros}
 \| Q_n^m(x,\cdot)-\pi_n\|_{\TV}\le   \sqrt{n} \| Q_n^m(x,\cdot)-\pi_n\|_{2}\leq  \sqrt{n} \sigma^{m}_{n,\bq},
\end{equation}
and we can conclude replacing $m$ by its value.
\end{proof}

\begin{remark}[Relaxation of the definition of the spectral radius.] \label{rq:flat2}
We may relax a little bit the assumption of strong convergence. If $H$ is a vector subspace of $\dR^{V_n}$ which is invariant under $P_{n,\bp}$, we set $\sigma_{n,\bp}^{H}$ to be the operator norm of $P_{n,\bp}$ on the orthogonal of $H$. Recall the definition of the flat-dimension $\dim_0$ in Remark \ref{rq:flat0}.

 Now, we  say that the sequence of actions $(\varphi_n)$ converges {\em relatively strongly} if for any finitely supported probability vector $\bp\in \ell^2 (\cG)$, we have $\limsup_n \sigma_{n,\bp} <1 $ and $\lim_n \sigma_{n,\bp} ^{H_n} = \sigma_{\bp}$ for a sequence $(H_n)$ of  invariant vector spaces such that $\pi_n \in H_n$ and $\dim_0(H_n) \leq n^{\veps_n}$ with $\lim_n \veps_n =0$. Then Theorem \ref{thegeneral} also holds under this weaker assumption. Indeed, we simply replace the bound  \eqref{spectros} by the bound valid for any  invariant vector space $H$ of $Q_n$ which contains $\pi_n$:
\begin{equation}\label{eq:l2+f}
\| Q_n^m ( x, \cdot)  - \pi_n \|_{\TV} \leq  \sqrt n \|Q_n^m(x,\cdot) -   \pi_n \|_{2} \leq \sqrt n \sigma_{n,\bq}^m \sqrt {\dim_0(H)/n} +  \sqrt n \left( \sigma^{H}_{n,\bq}\right)^m,
\end{equation}
which follows directly from the spectral theorem and the observation that, if $P_H$ is the orthogonal projection onto a vector space $H$, then, $$\|P_H f \|_2 \leq \sum_x |f(x)| \|P_H \ind_x \|_2 \leq \| f\|_1 \sqrt{\dim_0(H) / n}.$$ Finally, we notice that if $\dim_0(H)= n^ {o(1)}$ and $\limsup_n \sigma_{n,\bq} < 1$ then the first term on the right-hand side of \eqref{eq:l2+f} goes to $0$ as soon as $m$ is of order $\log n$.
\end{remark}

\section{Anisotropic random walks: proof of Theorem \ref{aniz}}

\label{sec:aniz}

\subsection{Notation} In this section, we fix an involution  as in Theorem \ref{aniz}. We define $\cG$ as the group obtained by $k$ free copies of $\dZ$ and $l$ free copies of  $Z/ 2 \dZ$ where $k+l$ is the number of equivalence classes of the involution, as detailed below Definition \ref{def:covering}.  We denote by $\cA = \{ g_1, \ldots, g_d\}$ its natural set of generators. The probability vector $\bp = (p_1, \ldots, p_d)$  as in Theorem \ref{aniz} is identified with a vector in $\ell^2 (\cG)$ defined by, for all $i\in [d]$, $p_{g_i} = p_i$ and $p_g = 0$ otherwise. As in the previous section, we denote by $(\cX_t)_{t \geq 0}$ the random walk on $\cG$ with transition kernel $\cP_{\bp}$ started at $\cX_0 = e$, the unit of $\cG$. The underlying probability distribution of the process $(\cX_t)_{t \geq 0}$ on $\cG^{\dN}$ will be denoted by $\dP (\cdot)$. Finally, we define  $\varphi_n$ as the action of $\cG$ on $V_n$ such that for all $i \in [d]$, $S_{g_i} = S_i$ where $S_i$ is as in \eqref{asym} and $S_{g}$ is the permutation matrix associated to $\varphi_n(g,\cdot)$.  Finally, given $x\in V_n$ we set $X_t = \varphi_n(\cX_t,x)$, that is $(X_t)_{t\geq 0}$ is a trajectory of the Markov chain on $V_n$  with transition kernel $P_{n,\bp}$ with initial condition $x$. 

\subsection{Proof strategy and organization}

Our starting point is to use the same stopping time strategy that for the previous section. But instead of using the RD property to conclude,  we are going to show that the generator of the backbone random walk can be 
reasonably approximated by a polynomial in $P_{n,\bp'}$, the generator of the random walk with anisotropy given by $\bp'$.
Our first job is thus to identify the value of $\bp'$ which is possible.
We perform this approximation for the backbone walk on the covering graph $\cG$ (it is then sufficient to use the covering to have an approximation for the walk on $V_n$). 
With the definition of the stopping set $U$ in \eqref{defuu}, a natural object to compare $\cQ$ to is the Green's operator which is expressed as a series in $\cP_{\bp}$. After a suitable truncation, we can in fact obtain a polynomial in  $\cP_{\bp}$ which is a good approximation of $\cQ$ .
By \textsl{good approximation in the $\ell^1$ sense}  we mean that one can find a polynomial $R$ which is such that 
\begin{equation}\label{poly}
\cQ(x,y) \le R(\cP_p)(x,y)
\end{equation}
 for all $x$ and $y$, and also such that 
 $\| R(\cP_p)(e,\cdot) \|_{\ell_1(\cG)}$ is not much larger than $\|\cQ(e,\cdot) \|_{\ell_1(\cG)}$ (which is equal to $1$). 
However, for our spectral computations, we want an approximation of $\cQ$ in the $\ell^2$ sense and it turns out that the above one is not satisfactory. 
In the same way that the Green's operator helps to find a good approximation in $\ell^1$, we want to use the operator  $\cR'_{\bp}$ defined by
$$ \cR'_{\bp}(x,y):= \sqrt{\cR_{\bp}(x,y)}$$
to find a good approximation of $\cQ\in \ell^2(\cG)$.
What makes this approach successful for anisotropic random walk on free groups is that $\cR'_{\bp}$ correspond to a point of the resolvent of \textsl{another anisotropic random walk $\cP_{\bp'}$ for a vector $\bp'$ which has the same support as $\bp$}. Again we can approximate the resolvent operator by a polynomial by an \textsl{ad-hoc} truncation procedure.

Our study of the resolvent of the random walk, presented in Section \ref{sec:resol}, allows us to derive an explicit relation between $\bp$  and $\bp'$.
Then in Section \ref{sec:deduce2} we show that this relation combined with a technical but somehow natural truncation proceedure yields
relevant bound on the kernel of the backbone walk (Proposition \ref{prop:dom}). Combining this with a few $\ell_2$ computations (Lemma \ref{comparaa}) this allows us to prove Theorem \ref{aniz} by adapting the approach used for Theorem \ref{thegeneral}.

\subsection{The relation between $\bp$ and $\bp'$ via resolvent} \label{sec:resol}

 The resolvent of $\cP_\bp$  is defined for $z \notin \sigma(\cP_{\bp})$ by
$$
\cR^z_\bp = ( z \cI - \cP_{\bp})^{-1}.
$$
 In the above expression, $\cI$ is the identity operator on $\ell^2 (\cG)$.
Since we are  particularly interested in the behavior of the operator $\cR^{z}_\bp$ as $|z| > \rho_{\bp}$ approaches $\rho_{\bp}$, we consider the following alternative definition  of $\cR^z_\bp(x,y)$ (which coincides with one above for $|z|>\rho_{\bp}$),
\begin{equation}\label{alteres}
 \cR^z_\bp(x,y):= \sum_{t\ge 0} z^{-(t+1)} \cP^t_{\bp}(x,y).
\end{equation}
As shown in \cite{MR824708} (see Lemma \ref{le:recres} below), the above series converges for all $x$ and $y$ if and only if $|z|\ge \rho'_{\bp}$ where 
$(\rho'_{\bp})^{-1}$ is the radius of convergence of the series $\cP^t_{\bp}(e,e)$. It is given by the following 
generalization of the Akemann-Ostrand formula (see \cite[Equation (2.6)]{MR824708}),
\begin{equation}\label{eq:AOnonrev}
\rho'_\bp = \min_{s > 0} \BRA{  2 s + \sum_{i=1}^d \PAR{ \sqrt{ s^2 + p_i p_{i^*} }  -s } },
\end{equation}
and satisfies $\rho'_{\bp}\le \rho_{\bp}$ (with equality in the symmetric case $p_i=p^*_i$ for all $i\in [d]$).

As our group is non-amenable, 
the vector $(\cR^1_{\bp}  (e ,x ))_{x \in \cG}$ is very close to be integrable 
(it does not belong to $\ell^1(\cG)$ but $(\cR^z_{\bp}  (e ,x ))_{x \in \cG}$ is in $\ell^{1} (\cG)$ for all $z > 1$), 
while $(\cR^{\rho_{\bp}}_{\bp} (e ,x ))_{x \in \cG}$ is close to be in $\ell^2(\cG)$ in the same sense.
What we prove in this section  (and which is made plausible by the observation above) is the following: 

\begin{proposition}\label{prop:ptop}
Given $\bp$ a probability vector on $[d]$ such that  \eqref{hippo} holds, there exists a unique probability vector $\bp'$ and a real $C = C (\bp)$ such that for all $x,y \in \cG$,
\begin{equation*}\label{propox}
\cR_{\bp}^1(x,y)= C \left(\cR^{\rho_{\bp'}}_{\bp'}(x,y)\right)^2.
\end{equation*}
\end{proposition}

To our knowledge, this quadratic identity has not been discovered before. It is of fundamental importance in what follows: it establishes a relation between the vector $(\cR^{1}_\bp (e,x))_{x \in \cG}$ - which, as seen in Section \ref{sec:deduce} is intimately connected with the entropy $\fh(\bp)$ - and $(\cR_{\bp'}^{\rho_{\bp'}}(e,x))_{x \in \cG}$, the resolvent of $\cP_{\bp'}$ at its spectral edge. As a consequence of the tree structure  of the Cayley graph  associated with $(\cG, \cA)$, 
that can be identified with the regular tree $\cT_d$,
the resolvent admits a simple ``multiplicative'' expression (this is a well established result that can be found e.g. in  \cite{MR824708} or \cite{MR1219707}). Indeed, $\cR^{z}_{\bp} ( e ,x )$ can be obtained by multiplying  $\cR^{z}_{\bp} ( e ,e)$  by a quantity $r^z_{i}(\bp)$ for each edge of type $i$ which is crossed on the minimal path linking $e$ to $x$. Hence to prove Proposition \ref{prop:ptop}, we need to find a probability vector $\bp'$ such that for all $i \in [d]$,
$r_i^1(\bp)=\left(r_{i}^{\rho_{\bp'}}(\bp')\right)^2$.

\medskip

We need some extra-notation to give an expression for the coefficients  $r_i$.
Let us denote by $\cR^{z}_{\bp,i}$ the resolvent of the operator $\cP_{\bp,i} = \cP_{\bp} -( p_i  \delta_{e} \otimes  \delta_{g_i} + p_{i^{*}} \delta_{g_i} \otimes \delta_{e})$ (defined as in \eqref{alteres} for $|z|\ge \rho'_{\bp}$) obtained from $\cP_{\bp}$ by removing the transitions between $e$ and $g_i$.  
Finally we let $\gamma^z_{i}$ be the diagonal coefficient of $\cR^z_{\bp,i} $:
\begin{equation}\label{gammasym}
\gamma^z_{i} = 
\gamma^z_{i}(\bp) :=  \cR^{z}_{\bp,i}(e,e)= \sum_{t\ge 0} \sum_{t\ge 0} z^{-(t+1)} \cP^t_{\bp,i}(e,e).  
\end{equation}
Note that  $\cP^t_{\bp}(e,e)$ is a function of $p_i p_{i^*}$, $i\in [d]$ (since every transitions from $e$ to $e$ involves the same number of multiplication by $g_i$ and $g_{i^*}$. It implies in particular that $\gamma^z_{i}=\gamma^z_{i^*}$. 
%  
% \begin{lemma}[see  Lemma 3 in \cite{MR824708} / Proposition 3.4 in \cite{MR1219707}]\label{le:recres}
% For any reduced word $x = g_{i_1}\dots {g_{i_n}}\in \cG$ written in reduced form (that is $g_{i_{k+1}} \ne g_{i_k^*}$ for all $k$) and $z \notin \sigma(\cP_{\bp})$, we have 
% $$
% \cR^{z}_{\bp} ( e ,x ) = \cR^{z}_{\bp} ( e ,e )\prod_{t=1}^n p_{i_t}\gamma^z_{i_t} . 
% $$
% Moreover, 
% \begin{equation}\label{eq:recres}
% \cR^{z}_{\bp} ( e ,e )  = \PAR{ z - \sum_{j\in [d]} p^2_j \gamma^z_j}^{-1} \AND \gamma^z_i  = \PAR{ z - \sum_{j \ne i^*} p^2_j \gamma^z_j}^{-1}.
% \end{equation}
% \end{lemma}
% 
% \hubert{VERSION NON REVERSIBLE 
Note that the reference in \cite{MR824708} only treats the case $q_2=d/2$  and assumes that every coordinates are positive. The positivity assumption however is not used in the proof (save for the fact that the return probability to zero decays exponentially, which is ensured by \eqref{hippo}). 
The proof also adapts to arbitrary values of $q_1$ and $q_2$  without any change
(cf. \cite[Proposition 3.4]{MR1219707} which only deals with the case $q_1=d$).\begin{lemma}[see  Lemma 2 and Lemma 3 in \cite{MR824708}, \cite{MR1219707}]\label{le:recres}
For any reduced word $x = g_{i_1}\dots {g_{i_n}}\in \cG$ written in reduced form (that is $g_{i_{k+1}} \ne g_{i_k^*}$ for all $k$) and $|z|\ge \rho'_{\bp}$ (recall \eqref{eq:AOnonrev})
$$
\cR^{z}_{\bp} ( e ,x ) = \cR^{z}_{\bp} ( e ,e )\prod_{t=1}^n p_{i_t}\gamma^z_{i_t} . 
$$
Moreover, 
\begin{equation}\label{eq:recresREV}
\cR^{z}_{\bp} ( e ,e )  = \PAR{ z - \sum_{j\in [d]}p_{j^{*}}  p_{j}\gamma^z_{j}}^{-1} \AND \gamma^z_i  = \PAR{ z - \sum_{j \ne i^*} p_{j^{*}} p_{j}\gamma^z_{j}}^{-1}.
\end{equation}
\end{lemma}
\noindent The above lemma allows to compute explicitly the resolvent operator. 

\begin{lemma}\label{le:sqz}
We assume that \eqref{hippo} holds and that $z \in [\rho'_{\bp},\infty)$. If $s=s_z = 1 / (2\cR^{z}_{\bp} ( e ,e  ))$ and $r_i= p_i \gamma^z_i$, we have $r_i r _{i^*} < 1$, 
$$
r_i = \frac{\sqrt{s^2 + p_i p_{i^*} } - s} {p_{i^*}} \,\text{ when }\, p_{i^*}>0,  \quad r_i = \frac{p_i}{2s}\,\text{ when }\, p_{i^*} = 0  \AND p_i = \frac{2 s r_i}{1-r_i r_{i^*}}.
$$
Moreover, $s_z $ is the largest real solution of the following equation in $x$
$$
 z =   2 x + \sum_{j = 1}^{d} \PAR{ \sqrt{ x^2 + p_j p_{j^*} }  -x}.
 % =  2 s \PAR{ 1 + \sum_j \frac{ q_j^2}{1- q_j^2}}. 
$$ 
\end{lemma}
\begin{proof}
From \eqref{eq:recresREV}, and the fact that $\gamma_i=\gamma_{i^*}$ (recall \eqref{gammasym}) we have $2s = z - \sum_j p_{j^*} r_j$ and $p_i p_{i^{*}}\gamma^2_i + 2s \gamma_i - 1 =0$.
The inequality $r_i r_{i^*} = p_i p_{i^{*}}\gamma^2_i <1$ and the formulas follow (also in the case $p_i p_{i^{*}}=0$). It remains to prove that $s$ is the largest solution of $f(x) = z$ with $f(x)= 
 2 x + \sum_{j} \PAR{ \sqrt{ x^2 + p_j p_{j ^*} }  -x } $. Since $f$ is strictly convex  and has a unique minimizer $x_{\min}\ge 0$ such that $f(x_{\min})=\rho'_{\bp}$. The equation $f(x)= z$ has either zero, one or two solutions according to whether $z<\rho'_{\bp}$, $z=\rho'_{\bp}$, $z>\rho'_{\bp}$. In the latter case, we let $x_- (z) <x_{\min}< x_+ (z)$ denote the two solutions. As $s_z$ is an increasing function of $z$ we have $s_z>x_{\min}$ for $z> \rho'_{\bp}$ and thus $s(z)=x_+ (z)$. \end{proof}

\begin{lemma}\label{le:qnormal}
Let $z \in [\rho_{\bp},\infty)$ and $r_i = p_i \gamma_i ^z$. We have
$$
\sum_{i=1}^d \frac{r_i(1-r_{i^*})}{1-r_ir_{i^*}} = 1   \Leftrightarrow z =1 \quad \quad
\text{and} 
\quad \quad
\sum_{i=1}^d \frac{r^2_i(1-r_{i^*}^2)}{1-(r_ir_{i^*})^{2}} = 1 \Leftrightarrow  z =\rho_\bp.  
$$
\end{lemma}

The result is a direct consequence of the following combinatorial statement (whose proof we include in the appendix for completeness).
\begin{lemma}\label{le:simplelemma}
 For any $(\alpha_i)_{i=1}^d$ in $[0,1)^d$, the function defined on $\cG$ by $F(x):=\prod_{t=1}^n \alpha_{i_t}$ if $x=g_{i_1}\dots g_{i_n}$ in reduced form, then $F$ is integrable for the uniform counting  measure on $\cG$ if and only if 
$$ \sum_{i=1} ^d \frac{\alpha_i(1-\alpha_{i^*})}{1-\alpha_i\alpha_{i^*}}<1.$$
 \end{lemma}

\begin{proof}[Proof of Lemma \ref{le:qnormal}]
From \eqref{alteres}, 
$\sum_{x\in \cG} \cR^z_{\bp}(e,x)<\infty $ if and only if $z> 1$.
On the other hand, from spectral considerations, $\|\cR_{\bp}^z \delta_e\|_2$ is finite for $z > \rho_{\bp}$ and diverges as $z$ goes to $\rho_{\bp}$. Hence, recalling definition \eqref{alteres}, we have $\sum_{x\in \cG} \left(\cR^z_{\bp}(e,x)\right)^2<\infty$ if and only if $z>\rho_{\bp}$.
Lemma \ref{le:simplelemma} implies that 
$$
\sum_{i=1}^d \frac{r_i(1-r_{i^*})}{1-r_ir_{i^*}} < 1   \Leftrightarrow z >1 \quad \quad
\text{and} 
\quad \quad
\sum_{i=1}^d \frac{r^2_i(1-(r_{i^*})^2)}{1-(r_ir_{i^*})^{2}} < 1 \Leftrightarrow  z >\rho_\bp.  
$$
We can conclude using the fact that (cf.  \eqref{gammasym}) the
$\gamma^z_i$s are continuous functions of $z$.\end{proof}
Now we are ready to identify the value of $\bp'$ which is such that \eqref{propox} holds. We set, for $i \in [d]$, $r_i^z(\bp)  = p_i \gamma_i ^z(\bp)$ and we introduce the vectors $\ba$ and $\bb$ whose coordinates  are given for all $i \in [d]$ by %$\bq^z(\bp) = (q_1, \ldots, q_d)$ with  and we set
$$a_i(\bp) := r^1_i(\bp) \quad \text{ and } \quad b_i(\bp) := r^{\rho_\bp}_i(\bp).$$
The formulas for the coordinates $a_i$ and $b_i$ of $\ba$ and $\bb$ are given in Lemma \ref{le:sqz} and Lemma \ref{le:qnormal} can be used to determine $\rho_{\bp}$  (this characterization of $\rho_{\bp}$ could also be deduced from  \cite[Corollary 3.1]{MR1844214}). Notably, by Lemma \ref{le:sqz} we have $a_i a_{i^*}, b_i b_{i^*} \in [0,1)$.  We can now reformulate and prove Proposition \ref{prop:ptop}.

\begin{proposition}\label{prop:ptoprime}
For any probability vector $\bp$ on $[d]$, there exists a unique probability vector $\bp'$ on $[d]$ with the same support as $\bp$  such that for all $i \in [d]$, 
we have 
$$a_i (\bp) = \left( b_i (\bp') \right) ^2.$$ 
It is given by the formula, for all $i \in [d]$,
$$
p'_i = \frac{\sqrt{ a_i (\bp)}}{1 - \sqrt{a_i(\bp)a_{i^*}(\bp)}}\left(\sum_{j\in [d]} \frac{\sqrt{ a_j (\bp)}}{1 -\sqrt{ a_j(\bp)a_{j^*}(\bp)}}\right)^{-1}. 
$$
\end{proposition}
\begin{proof}[Proof of Proposition \ref{prop:ptoprime}]
For ease of notation, we set $r'_i = \sqrt{a_i (\bp)}$. Assume that $\bp'$ is a probability vector such that $
b_i (\bp') = r'_i$ for all $i \in [d]$.  By Lemma \ref{le:sqz}, 
$(p'_i)_{i\in [d]}$ is the probability vector proportional to   $\left( r'_i / (1 - r'_i r'_{i^*}) \right)_{i\in [d]}$ and hence we have uniqueness. We now prove existence. We set $p'_i =  2sr'_i / (1 - r'_i r'_{i^*})$ where $s$
is the normalization constant such that $\bp'$ is a probability vector.  
Now setting  
\begin{equation}\label{tworootz}
  z =   2 s + \sum_{j = 1}^{d} \PAR{ \sqrt{ s^2 + p'_j p'_{j^* }}  -s},
\end{equation}
we only need to check that $s= 2/\cR^{z}_{\bp'}(e,e)$.
Indeed if this is the case, Lemma \ref{le:sqz} implies that $r'_i=r^z_i(\bp')$ and Lemma \ref{le:qnormal} implies that $z=\rho_{\bp'}$.
In view of \eqref{tworootz} and of the proof of Lemma \ref{le:sqz}, we only need to discard the possibility that
$s<x_{\min}$ where $x_{\min}$ is the minimizer of $f(x):= 2 x + \sum_{j = 1}^{d} \PAR{ \sqrt{ x^2 + p'_j p'_{j^* }}  -x}$.
Our definitions for $p'_i$ and $s$ imply that 
\begin{equation}\label{qprims}
r'_i = \frac{\sqrt{s^2 + p'_i p'_{i^*} } - s }{p'_{i^*}} \; \text{ if } p'_{i^*}>0
\quad \text{ and } \quad r'_i= \frac{p'_{i}}{2s} \; \text{ if  not}.
\end{equation}
Since both expressions above are monotone in $s$, if $s<x_{\min}$,
one would have 
$r'_i >  q^{\rho'_{\bp'}}_{i}(\bp') $ whenever $p'_{i}>0$ since $r^{\rho'_{\bp'}}_{i}(\bp)$ is obtained by substituting $s$ by $x_{\min}$ in  \eqref{qprims}  (here we use the definition \eqref{eq:AOnonrev} which implies that $\rho'_{\bp'}= f(x_{\min})$, as well as Lemma \ref{le:sqz}).
 As  $\rho'_{\bp'}\le \rho_{\bp'}$, this also implies that $r'_i >  r^{\rho_{\bp'}}_{i}(\bp')$ and thus that
\begin{equation}
 \sum_{i=1}^d \frac{(r'_i)^2(1-(r'_{i^*})^2)}{1-(r'_ir'_{i^*})^{2}}>1
\end{equation}
which yields a contradiction to the definition of $r'_i$. 
\end{proof}

\subsection{Deducing mixing time from a bounding kernel}
\label{sec:deduce2}

Our aim now is to work with the same stopping time and backbone walk as in Section \ref{skelex} and use the information we have to approximate  the transition matrix of the backbone walk $Q_{n}= P_{n,\bq}$  where the probability vector $\bq$ was defined below \eqref{eq:defQ}, with a power series of $P_{n,\bp'}$ the transition matridx of the nearest neighbor random walk associated with $\bp'$ of Proposition \ref{prop:ptoprime}. We further define $Q'_n$ to be the following truncated series (which approximates a multiple of the resolvent of $P_{n,\bp'}$ at $z = \rho_{\bp'}$)
$$Q'_n:= \frac{1}{\sqrt{k}}\sum_{t=0}^{\lfloor \log k\rfloor^4}  \left(\frac{P _{n,\bp'}}{\rho_{\bp'}} \right)^t.$$
(The fact that $\rho_{\bp'}$ is positive can be deduced from the expression \eqref{alteres} with a few simple computations. It also follows from the forthcoming Lemma \ref{le:entriesrho}).

\begin{proposition}\label{prop:dom}
Given $\bp$ a probability vector on $[d]$, there exists a real $C = C(\bp)$  such that for $\bp'$ given by Proposition \ref{prop:ptoprime}, we have, for all $x, y \in V_n$
\begin{equation*}
 Q_n (x,y)\le  C Q'_{n}(x,y)
\end{equation*}
\end{proposition}
We postpone the proof of this proposition to Section \ref{sec:post} and deduce 
Theorem \ref{aniz} out of it. The proof includes a few technical lemmas whose proofs are postponed to the end of this section.

\begin{proof}[Proof of Theorem \ref{aniz}]
Our first step is 
 to use the comparison above to obtain spectral estimates for $Q_n$.
We cannot control directly the spectral gap but we can estimate the contraction of functions with large variance.
More precisely, given a matrix $A$ of size $n \times n$ and $1 \leq u \leq \sqrt n$, we define 
\begin{equation}\label{eq:kappau}
\kappa_u (A):= \sqrt{\max_{f  :  \|f\|_2 \ge \frac{u}{\sqrt n } \|f\|_1  }\frac{\langle A f, A f \rangle }{\langle f,  f \rangle} }.
\end{equation}
Note that $\kappa_1 (A)$ is the operator norm of $A$ and  $\kappa_{\sqrt n}(A)$ is the square root of the maximal diagonal entry of $A^*A$. For general $u$, the scalar $\kappa_u(A)$ can be thought of as a kind of pseudo-norm of $A$ restricted to vectors which are localized in terms of their $\ell^2$ over $\ell^1$ ratio. The function $u \mapsto \kappa_u (A)$ can be thought of as a spectral analog (for a matrix) of the isoperimetric profile of a graph (if $A$ is the adjacency matrix of a graph, the isoperimetric profile is essentially obtained by restricting the maximum in \eqref{eq:kappau} to functions $f$ which are indicator functions of a subset of vertices).

\begin{lemma}\label{comparaa}
 Let $A,B$ be two $n \times n$ matrices such that $B$ is a bistochastic matrix. Assume that for some real $c \geq 0$ and   all $x,y \in [n]$, we have $|A(x,y)|\le c B(x,y)$ then for all $1 \leq u \leq \sqrt n $, 
 \begin{equation}\label{kappam}
  \kappa_u(A) \le   c \sigma(B)+ \frac{c}{ u},
 \end{equation}
 where $\sigma(B) = \| B_{|\IND^\perp} \|_{2 \to 2}$ is the  singular radius  of $B$.
 \end{lemma}

From Proposition \ref{prop:dom}, we may apply Lemma \ref{comparaa} when $A=Q_n$ and $B= \alpha Q'_n$, with $\alpha=k^{-1/2} \sum_{t=0}^{[\log k]^4} \rho_{\bp'}^{-t}$
and $c= C \alpha^{-1}$ for the constant $C$ given by Proposition \ref{prop:dom}.
In this case, from the triangle inequality, we have
\begin{equation}\label{eq:csigmaB}
 c \sigma(B)  = \NRM{\frac{C}{\sqrt{k}}\sum_{t=0}^{\lfloor \log k\rfloor^4}  \frac{(P^t _{n,\bp'})_{|\IND^\perp}}{\rho^t_{\bp'}} } \leq \frac{C}{\sqrt{k}} \sum_{t=0}^{\lfloor \log k\rfloor^4} \left( \frac{\sigma_{n,\bp'}(t)}{\rho_{\bp'}} \right)^{t},
\end{equation}
and, since $\rho_{\bp'} >0$ (see forthcoming Lemma \ref{le:entriesrho}), for some adequate choice of $C'$
\begin{equation*}
c = \frac{C}{\sqrt{k}} \sum_{t=0}^{\lfloor \log k\rfloor^4} \rho_{\bp'}^{-t}\le e^{C' (\log k)^4}.
\end{equation*}
We now bound \eqref{eq:csigmaB}. For that, we use the next proposition which quantifies the convergence of $\sigma_{\bp} (t)$ to $\rho_{\bp}$ in \eqref{eq:sympi}.
\begin{proposition}\label{prop:haag}
For any probability vector $\bp$ and integer $t \geq 1$, we have
$$
\rho_{\bp} \leq \sigma_{\bp} (t) \leq (t+1)^{2/t} \rho_{\bp}.
$$ 
\end{proposition}
 From Proposition \ref{prop:haag}, we deduce that 
 $$
\sum_{t=0}^{\lfloor \log k\rfloor^4} \left( \frac{\sigma_{\bp'}(t)}{\rho_{\bp'}} \right)^{t} \leq ( (\log k )^4 +1)^3,
 $$
Using Assumption \eqref{raminudge} and Lemma \ref{comparaa}, for any fixed $k\geq 5$, for all $n\ge n_0(k)$ sufficiently large, we obtain 
\begin{equation}\label{kappasmall}
\kappa_u(Q_n)\le  \frac{(\log k)^{13}}{\sqrt{k}}.
\end{equation}

Now we want to use this estimate to build an adapted time for the original walk $P_{n,\bp}$.
The idea is first to iterate $Q_n$ several times in order to contract the $\ell^2$ norm below the threshold $u$ and then use the original transition matrix $P_{n,\bp}$ to finish the job.
For this purpose, for a  large integer $k$ (we assume $k > (\log k)^{26}$)  which will be conveniently fixed later on, and for $n \geq 3$, we set
$$a_{n}:= \left\lfloor \frac{\log n}{\log k - 26 \log \log k}  \right \rfloor\quad  \text{ and } \quad   b_{n}:= \lfloor \log \log n \rfloor.$$
We define $T:=b_n+\tau_{a_{n}}$ where $(\tau_s)_{s \geq 0}$ are as in Subsection \ref{skelex} the successive times of the backbone walk.
Our spectral estimates \eqref{kappasmall} implies that $X_T$ is close to equilibrium:

\begin{lemma}\label{contrax}
For any fixed integer $k \geq 3$, let $a_n, b_n$ be as above and  $T=b_n+\tau_{a_{n}}$. If Assumption \eqref{raminudge} holds, then we have 
\begin{equation*}
 \lim_{n\to \infty} \max_{x\in V_n}\| \bbP_x[ X_{T} \in \cdot]- \pi_n \|_{\TV}=0.
\end{equation*}

\end{lemma}
To show that $$\max_{x\in V_n}\Tm_{n,\bp}(x,\gep)\le (1+ \delta) (\log n) / \fh(\bp)$$ for $n$ sufficiently large, we use Proposition \ref{stoop} with $t = t_n$ and $s = s_n$ where
$$T=b_n+\tau_{a_n}, \quad 
 t_n := \lfloor (1+ \delta/2) (\log n) / \fh(\bp) \rfloor \quad \text{ and } \quad  s_n:=\lfloor (\delta/2) (\log n) / \fh(\bp) \rfloor.$$ 
With this setup, the first term   in \eqref{en3}
tends to zero according to Lemma \ref{contrax}.  For  the third one we need to show that $\sigma_{n,\bp}(s_n)$ is bounded away from one.
Since \eqref{consequence} holds for $\bp'$ we have  (cf.  \eqref{lalimsup})
\begin{equation} \label{eq:koko}
\limsuptwo{n\to \infty}{s\to \infty} \sigma_{n,\bp'} (s) =\rho_{\bp'}<1.
\end{equation}
Now since  $\bp$ and $ \bp'$ have the same support one can compare $\sigma_{n,\bp'} (s)$ and $\sigma_{n,\bp'}(s)$. More precisely applying \cite[Lemma 13.22]{MR3726904} to the operators $\cP^s_{\bp}(\cP^{*}_{\bp})^s$ and $\cP^s_{\bp'}(\cP^{*}_{\bp'})^s$
yields for every $s$ and $n$
\begin{equation}
 \frac{1-\sigma_{n,\bp} (s)^{2s}}{1-\sigma_{n,\bp'} (s)^{2s}}\ge \min_{i\in [d]} \left(\frac{p_i}{p'_i}\right)^{2s}.
\end{equation}
Hence $\limsup_{n\ge 1}\sigma_{n,\bp} (s_0)<1$ for some $s_0$ and thus from \eqref{eq:sms} we get èthat
\begin{equation} \label{eq:koko2}
\limsuptwo{n\to \infty}{s\to \infty} \sigma_{n,\bp} (s) <1.
\end{equation}

It remains to show that 
$$\lim_{n\to \infty} \bbP[ \tau_{a_n}> t-b_n ]=0.$$
From the law of large number and Lemma \ref{le:tau}, for any $\delta>0$, we may choose 
an integer $k$ sufficiently large such that
\begin{equation*}
\lim_{n\to \infty} \bbP\left[ T \le \left( 1+ \frac\delta 4\right) a_{n}\frac{\log k}{\fh(\bp)}  \right] =1.
\end{equation*}

This concludes the proof of Theorem \ref{aniz}.
\end{proof}

 \begin{proof}[Proof of Lemma \ref{comparaa}]The statement is an immediate consequence the following functional inequality valid for every $f$
\begin{equation}\label{basic}
  \sqrt{\langle Af, A f \rangle }\le c \sigma(B)\|f\|_{2}+\frac{c}{\sqrt{n}}\|f\|_{1}.
\end{equation}
Since $B$ is bistochastic, the constant functions are left invariant by $B$ and its transpose. It follows that $\sigma(B)$ is the operator norm of $B$ projected on functions with zero sum. Now given $f$, if $|f|$ is the vector $|f|(x) := |f(x)|$ and $|A|$ is the matrix $|A|(x,y) := |A(x,y)|$, we have
  \begin{equation*}
   \langle Af, A f \rangle  \le \langle |A| |f|, |A| |f| \rangle \le c^2  \langle B |f|, B |f| \rangle.
  \end{equation*}
The orthogonal projection of $|f|$ on zero sums functions is  $\underline f (x) := |f|(x) - \| f \|_1 / n$. We have 
\begin{equation}\label{eq:Bff}
  \langle B |f|, B |f| \rangle = \|f\|^2_1 / n  +  \langle B \underline f, B \underline f \rangle
  \le  \|f\|^2_1 / n  + \sigma(B)^2 \| f \|^2_2.
\end{equation}
We deduce \eqref{basic} using the triangle inequality, $\sqrt{ a^2+b^2}\le |a|+|b|$.
\end{proof}

\begin{proof}[Proof of Lemma \ref{contrax}]
Recall $T = b_n + \tau_{a_n}$. The distribution of $X_{T}$ can be written as
$$\bbP_x[ X_T\in \cdot]=(P_{n,\bp}^{b_{n}}Q_n^{a_{n}})(x,\cdot).$$ 
We first show that for any $x \in V_n$ (recall $u=u_k:=e^{(\log k)^5}$)
 \begin{equation}\label{swoof}
 \|Q^{a_{n}}_n(x,\cdot)-\pi_n \|_2\le \frac{2u}{\sqrt n}.
 \end{equation}
Since $Q_n$ is a contraction, we note that  $\|Q_n^{t}(x,\cdot)-\pi_n  \|_2$ is non-decreasing in $t$. Moreover, 
$$
\|Q_n^{t+1}(x,\cdot)-\pi_n \|_2  =  \| Q_n ( Q_n^{t}(x,\cdot)-\pi_n  )  \|_2 \leq  \max\left( \kappa_u ( Q) \| Q_n^{t}(x,\cdot)-\pi_n \|_2 , \frac{2 u} {\sqrt n } \right), 
$$ where we have used that $ \| Q_n^{t}(x,\cdot)-\pi_n  \|_1 \leq 2$. Hence, an immediate induction yields for all $t \geq 0$,
\begin{equation*}
 \|Q_n^{t}(x,\cdot)-\pi_n \|_2\le \max\left( \kappa_u ( Q)^t , \frac{2 u} {\sqrt n } \right).
\end{equation*}
Thus, our bound \eqref{kappasmall} and our choice for $a_{n}$ imply \eqref{swoof}.
To conclude the proof, we use the usual $\ell^2$ bound and combine it with \eqref{swoof}. This gives 
\begin{equation}
    \|  P_{n,\bp}^{b_{n}} Q_n^{a_{n}}(x,\cdot)- \pi_n \|_{\TV}   \leq \frac {\sqrt n}{2}  \|  P_{n,\bp}^{b_{n}} Q_n^{a_{n}}(x,\cdot)- \pi_n\|_2\le \frac {\sqrt n}{2}  \sigma_{\bp,n}(b_n)^{b_{n}}   \| Q^{a_n}(x,\cdot)- \pi_n \|_2 \le  \sigma_{\bp,n}(b_n)^{b_{n}}   u .
\end{equation}
Finally we conclude by using \eqref{eq:koko} and that $b_n$ tends to infinity.
\end{proof}

\begin{remark}[Relaxation of our assumption concerning  the spectral radius] \label{rq:flat1}
As in Remark \ref{rq:flat0}, we denote by $\dim_0(H)$  the flat-dimension of a vector space $H$ of $\dR^{V_n}$ and we set $\rho_{n,\bp}^{H}$ to be the operator norm of $P_{n,\bp}$ on the orthogonal of $H$. We may modify  Theorem \ref{aniz} as follows: if $(H_n)$ is a sequence of invariant vector spaces of $P_{n,\bp'}$ such that $\lim_n \rho_{n,\bp'} ^{H_n} = \rho_{\bp'}$ and $\dim_0(H_n) \leq n^{o(1)}$ (that is $\lim_n \log \dim_0(H_n)/\log n=0$) then the conclusion  of Theorem \ref{aniz} holds.

Indeed, in Lemma \ref{comparaa}, if $H$ is an invariant subspace of the bistochastic matrix $B$ and its transpose, then \eqref{kappam} can be improved to
$
\kappa_u (A) \leq c \rho_H(B) + c \sqrt{\dim_0(H)} / u,
$
 where  $\rho_H(B)$ is the operator norm of $B$ on the orthogonal of $H$. Recall that if $P_H$ is the orthogonal projection onto $H$, then $\|P_H g \|_2 \leq \| g\|_1 \sqrt{\dim_0(H) / n}$. Setting $\underline g = |f| - P_H|f|$, we may thus replace the bound \eqref{eq:Bff} by  $\langle B |f| , B |f| \rangle  \leq \| f \|^2_1 \dim_0(H)  / n + \langle B \underline g , B \underline g \rangle   \leq \| f \|^2_1 \dim_0(H)  / n +  \rho_H(B) \|f\|^2_2$. It gives the claimed improvement of \eqref{kappam}. The rest of the argument is essentially unchanged (the sequence $b_n$ has to be chosen so that $\veps_n \log n \ll b_n \ll \log n$). \end{remark}

\begin{remark}[More quantitative bound on the mixing time]
A more quantitative upper bound on $\Tm_n(\gep)$ can be obtained by choosing $k_n$ tending to infinity, and using a more quantitative version of Proposition \ref{lb} for anisotropic walks on trees. In the reversible case \eqref{eq:sympi}, optimizing all choices of parameters in our proof, we obtain a result of the form
$$\Tm_{n,\bp}(\gep)\le \frac{\log n}{\fh(\bp)}+ C(\log n)^{2/3}$$
provided that $\rho_{n,\bp'}$ converges fast enough to $\rho_{\bp'}$ as $n$ go to infinity (more specifically we require $\rho_{n,\bp'}\le \rho_{\bp'}+ C (\log n)^{-1/3} $).
Note that our correction term is larger than $(\log n)^{1/2}$, 
and thus the proof developed in this section does not allow to obtain the anisotropic counter-part of Equations \eqref{lupper}, \eqref{profyle}, which allow to describe more accurately the profile of relaxation to equilibrium provided some quantitative information about the convergence \eqref{raminudge} is given.
\end{remark}

\subsection{Proof of Proposition \ref{prop:haag}}

We start with a general lemma on the spectral radius of the operator $\cP_{\bp}$ and the probability of transitions.
\begin{lemma}\label{le:entriesrho}
Let $\cG$ be a finitely generated group. For any probability vector $\bp \in \ell^2(\cG)$, any integer $t \geq 1$, any $x \in \cG$, we have 
$$
\| \cP_\bp^t \delta_x \|_2  \leq \rho^t_{\bp}.
$$
\end{lemma}
\begin{proof}
We may assume $x = e$ without loss of generality.
Since $\cP_\bp^t$ is the generator of a random walk with spectral radius $\rho^t_{\bp}$ we may also assume that $t=1$.
We have that $\| \cP_\bp \delta_e \|_2 \leq \sigma_{\bp}(1) = \| \cP_{\bp} \|_{2 \to 2}$. If the reversibility condition \eqref{eq:sympi}  holds, then $\sigma_{\bp}(1) = \rho_{\bp}$ and the lemma follows. In the general case, we use the group structure to obtain the required bound. We first write that for any integer $k \geq 1$,
$$
\| \cP_\bp \delta_e \|^{2k}_2 = \left(\sum_{x \in \cG} \cP_\bp (e,x)^2  \right)^k = \sum_{x_1,\ldots, x_k} \left( \cP_\bp (e,x_1)\cdots \cP_\bp (e,x_k) \right)^2.
$$
Using that $\cP_\bp (xg,yg) = \cP_\bp(x,y) $ for all $x,y,g$ in $\cG$, we may write
$$
\cP_\bp (e,x_1)\cdots \cP_\bp (e,x_k)  = \cP_\bp  (e,x_1) \cP_\bp (x_1,x_2x_1) \cdots \cP_\bp (x_{k-1} \cdots x_{1},x_k\cdots x_{1}),
$$
and
\begin{align*}
\sum_{x \in \cG} \left(\cP^k_{\bp} (e,x)  \right)^2 & =  \sum_{x} \left(\sum_{x_1,\ldots,x_{k-1}}\cP_\bp  (e,x_1) \cP_\bp (x_1,x_2x_1) \cdots \cP_\bp (x_{k-1} \cdots x_{1},x) \right)^2 \\
& \geq  \sum_{x}  \sum_{x_1,\ldots,x_{k-1}} \left(\cP_\bp  (e,x_1) \cP_\bp (x_1,x_2x_1) \cdots \cP_\bp (x_{k-1} \cdots x_{1},x) \right)^2 \\
& =  \sum_{x_1,\ldots,x_{k}} \left(\cP_\bp  (e,x_1) \cP_\bp (x_1,x_2x_1) \cdots \cP_\bp (x_{k-1} \cdots x_{1},x_k \cdots x_1) \right)^2.
\end{align*}
We deduce that 
$$
\| \cP_\bp \delta_e \|^{2k}_2 \leq \| \cP^k_\bp \delta_e \|^{2}_2  \leq \sigma_{\bp} (k)^{2}.
$$
We now let $k$ tend to infinity and apply \eqref{eq:gelfand}.
\end{proof}
Proposition \ref{prop:haag} is now an immediate consequence of the RD property \eqref{eq:RD} for the free group.
\begin{proof}[Proof of Proposition \ref{prop:haag}]
Haagerup inequality (that is, RD property for free groups) implies that for any $t \geq 0$,
$$
\sigma_{\bp} (t)^t = \| \cP_{\bp} ^t \|_{2 \to 2} \leq (t+1)^2 \| \cP_{\bp}^t \delta_e\|_2,
$$
see \cite[Lemma 1.4]{MR520930} (the proof is written in the case of the free group, denoted by $\cG_{\mathrm{free}}^{d,0}$ with our notation, but applies also to $\cG_{\mathrm{free}}^{q_1,q_2}$ with $q_1 + 2 q_2 = d$). It remains to use Lemma \ref{le:entriesrho}.\end{proof}

\subsection{Proof of Proposition \ref{prop:dom}}\label{sec:post}
 The matrices $Q_n$ and $P_{n,\bp'}$ are both defined as the transition kernel corresponding to projections of Markov chains on the group $\cG$ on $V_n$. From \eqref{eq:defPnp}-\eqref{eq:defPp}, if $\bq$ is a finitely supported probability vector on $\cG$, for all $x,y$ in $V_n$,
 $$
P_{n,\bq} (x,y) = \sum_{g\in \cG} \cP_{\bq} (e,g) \IND (  \varphi_n (g,x) = y),   
 $$
where $\varphi_n$ is the action of $\cG$ on $V_n$. It is thus sufficient to prove the inequality for the corresponding kernels  $\cQ$ (as in \eqref{eq:defQ}) and $\cP_{\bp'}$ on $\cG$, that is 
\begin{equation}\label{domination}
\forall x\in \cG, \quad  \cQ (e,x)\le  \frac{C}{\sqrt{k}}\sum_{t=0}^{\lfloor \log k\rfloor^4}  \left(\frac{\cP_{\bp'}}{\rho_{\bp'}} \right)^t (e,x).
\end{equation}
Since $\cQ (e,x) = 0$ for all $x \notin \partial U$, it is sufficient to check \eqref{domination} for $x \in \partial U$.  By Lemma \ref{le:recres},  if $z \geq \rho_{\bp}$ and $x = g_i y$ for some $g_i \in \cA$, then $\cR^z(e,x) \geq c \cR^z(e,y)$ for some positive $c = c(\bp,z)$. Since  $\cR^1(e,y) \geq 1/k$ for all $y \in U$, we find for all $x \in \partial U$,
$$
\cQ (e,x) \leq \frac 1  k \leq \frac{C}{\sqrt k} \sqrt{\cR_{\bp}^1 (e,x) },
$$
with $C = 1/ \sqrt c$. Thus, from Proposition \ref{prop:ptop}, for some new constant $C = C (\bp)$, for all $x \in \partial U$,
$$
\cQ (e,x) \leq  \frac{C}{\sqrt k}  \cR_{\bp'}^{\rho_{\bp'} } (e,x) .  
$$
To deduce \eqref{domination} from this last bound, we expand the resolvent as a power series. It requires some care because,  when the reversibility condition \eqref{eq:sympi} holds, $z=\rho_{\bp'}$ is precisely the  threshold $\rho'_{\bp}$ for the convergence the power series \eqref{alteres}.

With the notation of Lemma \ref{le:recres}, for any $\bp$ and $i \in [d]$, the function $z \mapsto \gamma_i ^z$ is decreasing in $z \geq \rho_\bp$. Moreover, by Lemma \ref{le:sqz}, using the strict convexity of the function $f$ there, 
we have for all $z \geq \rho_\bp$, $\gamma_i^{\rho_\bp} - \gamma_i ^z \leq C \sqrt{ z - \rho_{\bp}}$ for some $C = C(\bp)$ (the inequality is even valid without square-root when $\rho_{\bp}<\rho'_{\bp}$ for an adequate choice of constant). 
 By Lemma \ref{le:recres}, it follows that for some new $C = C(\bp)$ for all $x \in \cG$, 
$$
|\cR^{\rho_\bp}_\bp (e,x) - \cR^{z}_\bp (e,x) | \leq C (|x| +1) \sqrt{ z - \rho_{\bp}}\cR^{\rho_p}_\bp (e,x),
$$
where $|x|$ is the distance of $x$ to $e$ in the tree $\cT_d$ and where we have used the telescopic sum decomposition (with the convention that a product over an empty set is one)
$$\prod_{i=1}^k a_i - \prod_{i=1}^k b_i = \sum_{j=1}^k \left(\prod_{i=1}^{j-1} a_i \right)( b_j - a_j) \left( \prod_{i=j+1}^{k} b_i \right).$$
By Lemma \ref{le:UnifU}, the diameter of $\partial U$ being at most $C \log k$, we find that for all $x \in \partial U$, 
$
\cR^{\rho_\bp}_\bp (e,x)  \leq  2 \cR^{z}_\bp (e,x) 
$ provided that $0 \leq z - \rho_{\bp} \leq c (\log k)^{-2}$ for some positive constant $c = c(\bp) >0$.  
We now fix $z = \rho_{\bp'} + c(\bp') (\log k)^{-2}$. From what precedes, for all $x \in \partial U$, 
$$
 \cR_{\bp'}^{\rho_{\bp'} } (e,x) \leq 2 \cR^{z}_{\bp'} (e,x) = \frac{2}{z} \sum_{t=0}^{\infty}\left( \frac{\cP_{\bp'}}{z}\right) ^t(e,x). 
$$
By Lemma \ref{le:entriesrho}, we have $\cP_{\bp'}^t (e,x) \leq \rho_{\bp'}^t$ and, for some new constant $C = C(\bp')$, for any $s \geq 0$, 
$$
\sum_{t = s}^{\infty} \left( \frac{\cP_{\bp'}}{z}\right) ^t(e,x)\leq C  (\log k )^2 e^{- \frac{s}{C (\log k)^2}}.
$$
We now recall that by Proposition \ref{prop:ptop}, for all $x \in \partial U$, $\cR^{z}_{\bp'} (e,x) \geq \cR^{\rho_{\bp'}}_{\bp'} (e,x) /2 \geq c / \sqrt k$. It follows that if $s = \lfloor C' \log k\rfloor^3$ for some large enough constant $C'$, we have 
$$
\frac 1 { z} \sum_{t = s}^{\infty} \left( \frac{\cP_{\bp'}}{z}\right) ^t(e,x) \leq \frac 1 2  \cR^{z}_{\bp'} (e,x). 
$$ 
Consequently, for this value of $s$, 
$$
\cR^{z}_{\bp'} (e,x) \leq  \frac 2 { z}  \sum_{t=0}^{s}\left( \frac{\cP_{\bp'}}{z}\right) ^t(e,x) \leq  \frac 2 { \rho_{\bp'}}  \sum_{t=0}^{s}\left( \frac{\cP_{\bp'}}{\rho_{\bp'}}\right) ^t(e,x).
$$
This concludes the proof of \eqref{domination}
\qed

\section{Random walks covered by a colored group}

\label{sec:lifted}

\subsection{Minimal mixing time for color covered random walks}
We now present a last extension of our results.  As in the setting of Theorem \ref{thegeneral}, we assume that for a finitely generated non-amenable group $\cG$, we have a sequence of finite sets $(V_n)$ with $\# V_n = n$ and $(\varphi_n)$ a sequence of actions of $\cG$ on $V_n$.  Let $r \geq 1$ be an integer. We think of $[r] = \{1, \ldots, r\}$ as a set of colors.  An element $\bp \in M_r (\dR) ^{\cG}$ is written as a matrix-valued vector $\bp = (p_g)_{g \in \cG}$ with $p_g \in M_r (\dR)$. The support of $\bp$ is then the subset of $\cG$ such that $p_g$ is not the null matrix. We consider $\bp  \in M_r (\dR) ^{\cG}$  with finite support such that
$$P_{1,\bp} := \sum_{ g \in \cG} p_g$$ is an irreducible stochastic  matrix on $[r]$ 
with invariant probability measure $\mu$. % and which satisfies the matrix-valued analog of \eqref{eq:sympi}: 
%\begin{equation}\label{eq:balance}
%\begin{cases}
%&\text{The support of $\bp$ generates $\cG$ } \\
%&\forall\, (u,v) \in [r]^2,  \forall g \in \cG, \quad \mu(u) p_g (u,v) = \mu(v) p_{g^{-1}} (v,u).\end{cases}
%\end{equation}
%In particular, $P_{1,\bp}$ is reversible with respect to $\mu$. 
Then, we denote by $\cP_{\bp}$ the operator on $\ell^2 (\cG \times [r])$ defined by 
 \begin{equation}\label{eq:defPpQ}
\cP_{\bp}=\sum_{g \in \cG }p_g \otimes \lambda(g),
\end{equation}
where $\lambda(g)$ is as in \eqref{eq:defPp} and $\otimes$ is the tensor product.  In probabilistic terms, $\cP_{\bp}$ is the transition kernel of a random walk $(\cX_t)$ on $\cG \times [r]$ where the probability to jump from $(g,u)$ to $(h,v)$ is $p_{h g^{-1}} (u,v)$. We denote by $\rho_{\bp}$ the spectral radius of $\cP_\bp$ and by $\fh(\bp)$ the entropy rate of $\cP_{\bp}$ defined by: for any $u_0 \in [r]$, 
$$
\fh(\bp) = \lim_{t \to \infty}  - \frac{1}{t}\sum_{(g,u) \in \cG\times [r]} \cP_{\bp}^t  ((e,u_0),(g,u)) \log \cP_{\bp}^t((e,u_0), (g,u)). 
$$
The fact that $\fh(\bp)$ does not depend on $u_0$ is an easy consequence of the assumption that $P_{1,\bp}$ is irreducible. Again, if $\cG$ is non-amenable  and $\rho_{\bp} < 1$   holds, then  $\fh(\bp) >0$. % and $\rho_{\bp}  < 1$, where $\rho_{\bp}$ is the spectral radius of $\cP_{\bp}$.
 Besides, the proof of Shannon-McMillan-Breiman Theorem in \cite[Theorem 2.1]{MR704539}, actually proves that  if $ \cX_0 = (e,u_0)$, a.s.
 \begin{equation}\label{eq:SMMBQ}
 \fh (\bp) = \lim_{t \to \infty}  - \frac{\log \cP_{\bp}^t ((e,u_0),  \cX_t) }{t} .
 \end{equation}
With $(S_g)_{g \in \cG}$ as in \eqref{eq:defPnp}, we define the stochastic matrix  on $\dR^{V_n \times [r]}$  
\begin{equation}\label{eq:defPnpQ}
P_{n,\bp} = \sum_{g\in \cG} p_g \otimes S_g.
\end{equation}
This matrix is the transition kernel a  Markov chain on $V_n \times [r]$ covered by $( \cX_t)$ in the sense that if we define for $ (g,u ) \in \cG \times [r]$ and $x \in V_n$, $ \bar \varphi_n( (g,u) , x) := (\varphi_n (g,x),u)$ then $X_t := \bar \varphi_n(  \cX_t,x)$ is a Markov chain with transition matrix $P_{n,\bp}$ started at $(x,u_0)$. The measure $\pi_n (x,u) = \mu (u)/n$ is an invariant probability of $P_{n,\bp}$.
Moreover, since \eqref{eq:SMMBQ} holds, the proof of Proposition \ref{lb} actually implies that mixing time of $ X_t$, $\Tm_{n,\bp}(x,\veps)$, satisfies for any fixed $\veps \in (0,1)$ and uniformly in $x\in V_n$, the  lower bound \eqref{lentroop}.

\medskip

This setting allows to consider a random walk on the {\em $n$-lift of a base graph}. More precisely, let $G_1$ be a finite simple connected graph with $d/2$ undirected edges on the vertex set $[r]$. We consider the free group $\cG_{\mathrm{free}}$ with $d/2$ generators and their $d/2$ inverses $(g_1, \ldots, g_d)$, that is $g_i^{-1} = g_{i^*}$ for some involution on $[d]$ without fixed point.  Each generator $g_i$ is associated to a directed edge $(u_i,v_i)$ of $G_1$ and $g_i^{-1} = (v_i,u_i)$ is the inverse directed edge.  We consider the action of $\cG_{\mathrm{free}}$ on $[n]$ defined by $\varphi_n (g_i,x) = \sigma_i (x)$ where $(\sigma_1, \ldots, \sigma_d)$ are permutation matrices such that $\sigma_i^{-1} = \sigma_{i^*}$. Then, if $E_{k,\ell}  \in M_{r}(\dR)$ is the canonical matrix defined by $E_{k ,\ell}(i,j) = \IND_{\{(k,\ell)= (i,j)\}}$, then the graph $G_n$ with vertex set $[n] \times [r]$ and adjacency matrix $\sum_i E_{u_i,v_i} \otimes S_i$ is a simple graph which is called a $n$-lift (or a $n$-covering) of $G_1$: the $[n] \times [r] \to [r]$ map $\psi(x,u) = u$ is $n$ to $1$ and, for any $(x,u)$, the image by $\psi$ of the adjacent vertices of $(x,u)$ in $G_n$ coincides with the adjacent vertices of $\psi(x,u)$ in $G_1$. If $d_u$ is the degree of the vertex $u$ in $G_1$ and $p_{g_i} = E_{u_i,v_i} / d_{u_i} $ then $P_{1,\bp}$ and $P_{n,\bp}$ are the transition matrices of the simple random walks on $G_1$ and $G_n$ respectively. %In this case the condition \eqref{eq:balance} is fulfilled with $\mu(u) = d_{u} / d$.

\medskip

We are ready to state the analog of Theorem \ref{thegeneral}.

\begin{theorem}\label{thegeneralQ}
Let $\cG$ be a finitely generated non-amenable group with the property RD,  $(V_n)$  a sequence of finite sets with $\# V_n = n$ and $(\varphi_n)$ a sequence of  actions of $\cG$ on $V_n$ which converges strongly.  For any integer $r \geq 1$ and any finitely supported $\bp \in M_r (\dR)^{\cG}$  such that $\rho_{\bp} < 1$  and  $P_{1,\bp}$ is an irreducible aperiodic Markov chain,  
the mixing time of the random walk with transition matrix $P_{n,\bp}$  satisfies, for every $\gep\in (0,1)$,
 \begin{equation*}
  \lim_{n\to \infty} \frac{\Tm_{n,\bp}(\gep)}{\log n}=\frac{1}{\fh(\bp)}.
 \end{equation*}
\end{theorem}

Note that in the above statement the RD property and the strong convergence property are defined in terms of scalar valued vectors $\bp \in \ell^2 (\cG)$.  From \cite{BC18}, an example of application of Theorem \ref{thegeneralQ} is the random walk on a random $n$-lift of a weighted base graph such that  $P_{1,\bp}$ is irreducible and aperiodic (see \cite{C-K} for a recent alternative and independent proof of this case).

\subsection{Proof of Theorem \ref{thegeneralQ}} We let $(\cX_t)$ be the random walk with kernel $\cP_\bp$ started from $ \cX_0 = (e,u_0)$. For $ (g,u ) \in \cG \times [r]$ and $x \in V_n$, we set $  \bar \varphi_n( (g,u) ,x) = (\varphi_n (g,x),u)$ and let $X_t := \bar  \varphi_n(  \cX_t,x)$ be a Markov chain with transition matrix $P_{n,\bp}$ started at $(x,u_0)$.   We adapt the arguments of Section \ref{sec:thegeneral} to our matrix-valued context. 

\subsubsection{Relative spectral radius, strong convergence and RD property}
Let $\bq = (q_g) \in M_r (\dR)^{\cG}$ with finite support. We define $\ell^2(\mu)$ as the Hilbert space on  $\dR^r$ endowed with the scalar product $\langle f , g \rangle_{\mu}  = \sum_{i} \mu(i)\bar f(i) g(i)$. Similarly, $\ell^2_{n}(\mu)$ and $\ell^2_{\cG} (\mu)$ are the Hilbert spaces on the vector spaces $\dR^{V_n \times [r]}$ and $\dR^ {\cG \times [r]}$  endowed with the scalar products:
$$
\langle f , g \rangle_{\mu}  = \sum_{(x,i) \in X \times [r]} \mu(i)\bar f(x,i) g(x,i),
$$
with $X = V_n$ and $X = \cG$ respectively.  We note that the subspace of $\dR^ {V_n \times [r]}$: $H_r = \dR^r \otimes \ind$  of vectors $f$ of the form for some $g \in \dR^r$, $f(x,i) = g(i)$ is an invariant subspace of dimension $r$ for $P_{n,\bq}$  and its adjoint in $\ell^2_n (\mu)$.  Hence $P_{n,\bq}$ admits a direct sum decomposition  on  $H_r \oplus H_r^\perp$. We note also the restriction of $P_{n,\bq}$ to $H_r$ coincides with $P_{1,\bq}$. %It follows that all eigenvalues of $P_{1,\bq}$ are also in the spectrum of $P_{n,\bq}$.
We define the {\em relative singular radius} as the following operator norm
\begin{equation}\label{eq:specradQ}
\bar \sigma_{n,\bq} := \| (P_{n,\bq}) _{|H_r^\perp} \|_{\ell^2_{n}(\mu) \to \ell^2_{n}(\mu)} .
\end{equation}
From \cite[p256]{MR3585560} (see also \cite{MR1401692}), if $(\varphi_n)$ converges strongly then we have 
\begin{equation}\label{eq:SCQ}
\lim_{n \to \infty} \bar \sigma_{n,\bq} = \sigma_{\bq},
\end{equation} 
where $ \sigma_{\bq} := \| \cP_{\bq} \|_{\ell^2_{\cG} (\mu) \to \ell^2_{\cG} (\mu)} $. 

Besides, let $E_{ij} \in M_r (\dR)$ be the canonical matrix with all entries zero but entry $(i,j)$ equal to $1$. The $\ell^2(\mu) \to \ell^2(\mu)$ operator norm of $E_{ij}$ is $\sqrt{ \mu(i) / \mu(j)}$. Since $\rho_{\bq}$ coincides with the  $\ell^2_{\cG}(\mu) \to \ell^2_{\cG}(\mu)$ operator norm, from the triangle inequality, we have 
$$
\sigma_{\bq}  = \NRM{ \sum_{{i,j} \in [r]^2 } \sum_{g \in \cG} q_g(i,j) E_{ij} \otimes  \lambda(g) }_{\ell^2_{\cG}(\mu) \to \ell^2_{\cG}(\mu)}   \leq   \sum_{(i,j) \in [r]^2}\sqrt{\frac{\mu(i)}{ \mu(j)}} \sigma_{\bq(i,j)},
$$
where  $\bq(i,j)  = (q_g(i,j)) \in \dR^{\cG}$ and $\sigma_{\bq(i,j)}$ is the singular radius of  $\cP_{\bq(i,j)}$ in $\ell^2(\cG)$. It follows, that if $\cG$ has the RD property and $R$ is the diameter of the support of $\bq$ (in the Cayley graph associated to any symmetric generating set $\cA$) then, for some constant $C(\cG,\cA)>0$,
\begin{equation}\label{eq:RDQ}
\sigma_{\bq}  \leq C R^C \sum_{(i,j) \in [r]^2} \sqrt{\frac{\mu(i)}{ \mu(j)}} \sqrt{ \sum_{g \in \cG} q_g(i,j)^2} \leq  Cr^2 R ^C \sqrt{ \sum_{g \in \cG} \| q_g \|_{\ell^2(\mu)\to \ell^2(\mu)}^2},
\end{equation}
where we have used that  $\sqrt{\mu(i)/\mu(j)} |q_{g}(i,j)| = \| q_g (i,j) E_{ij} \|_{\ell^2(\mu)\to \ell^2(\mu)} \leq \| q_g \|_{\ell^2(\mu)\to \ell^2(\mu)}$.

\subsubsection{Skeleton Walk.} We now adapt the argument of Subsection \ref{skelex}.
We let $\cR_{\bp}= (\cI_{\cG \times [r]} -  \cP_{\bp} )^{-1}$ be the Green's operator associated with $\cP_{\bp}$. 
For $g,h \in \cG$, we denote by $\cR_{\bp} (g,h) \in M_r(\dR)$ the matrix whose entry $(i,j)$ is $\cR_{\bp} ((g,i),(h,j))$. 
 For $g \in \cG$, we define $u(g) \in M_r (\dR)$ as the matrix 
$$
u(g) :=  \cR_{\bp} (e,g) = \sum_{t = 0}^ {\infty} \cP_{\bp}^t (e,g),
$$
where  $\cP_{\bp}^t (g,h)  \in M_r(\dR)$ has entry $(i,j)$ equal to  $\cP^t_{\bp} ((g,i),(h,j))$.
 Given $k \geq 1 $, we define the set 
\begin{equation}\label{defuuQ}
\cU := \BRA{ g \in \cG : \| u (g) \|_{\ell^2(\mu)\to \ell^2(\mu)} \geq 1/k}. 
\end{equation}
The backbone walk is the induced walk on the successive exit times from $\cU$:  $\tau_0 := 0$, $\tau_1 = \tau$ and, for integer $s \geq 1$, $\tau_{s+1} := \inf \{ t \geq \tau_k :  \cX_{t}  \cX^{-1}_{\tau_s} \notin \cU \}$. We define $\cQ = \cP_{\bq}$ as the transition kernel of the random walk $ \cX_{\tau_m}$.

From \eqref{eq:SMMBQ} and $\rho_{\bp} < 1$, the proofs and statements of Proposition \ref{le:UnifU} and Lemma \ref{le:tau} continue to hold in our new setting (in \eqref{eq:tandilt}, we replace $q_g \leq 1/k$ by $\| q_g \|_{\ell^2(\mu)\to \ell^2(\mu)} \leq 1 /k$).

\subsubsection{Deducing mixing time from RD property and the strong convergence}
We may now conclude the proof of Theorem \ref{thegeneralQ} by adapting the content of Subsection \ref{sec:deduce}.

\begin{proof}[Proof of Theorem \ref{thegeneralQ}]
We fix $\veps \in (0,1)$, $\delta>0$  and $(x,u_0) \in V_n \times [r]$ and prove that for $n$ sufficiently large
$$\Tm_{n,\bp}((x,u_0),\gep)\le   (1+\delta)\log n /\fh(\bp).$$ 
Let $(\tau_m)$ and $U$ be as above for some large $k$ to be chosen. We set   $m:= \lfloor (1+\delta/4) (\log n ) / \log k \rfloor$. 

For integer $s \geq 1$, the relative $s$-th singular radius is 
$$
\bar \sigma_{n,\bp} (s) := \| (P^s_{n,\bp}) _{|H_r^\perp} \|^{1/s}_{\ell^2_{n}(\mu) \to \ell^2_{n}(\mu)} \quad \hbox{ and } \quad \sigma_{\bp} (s) := \| \cP^s_{\bp} \|_{\ell^2_{\cG}(\mu) \to \ell^2_{\cG}(\mu)}^{1/s}.
$$
From \eqref{eq:SCQ}, for all $s \geq 1$, $\lim_n \bar \sigma_{n,\bp} (s) = \sigma_{\bp}(s) < 1$. Since $\rho_{\bp} < 1$ and $\lim_{s \to \infty} \sigma_{\bp} (s) = \rho_\bp$, we deduce that for all $s \geq s_0$ large enough and all $n \leq n_0$ large enough, $\sigma_{n,\bp} (s) \leq 1- \delta_0$ for some $\delta_0 >0$ (we argue as below \eqref{eq:sms}). Moreover, since  $P_{1,\bp}$ is irreducible and aperiodic, we have that $\sigma_{1,\bp} <1$. We deduce that for $\delta = \min( \delta_0, 1 - \sigma_{1,\bp} ) >0$, for all $s \geq s_0$ and all $n \geq n_0$ 
\begin{equation}\label{eq:gdidie}
\sigma_{n,\bp} (s) = \| (P^s_{n,\bp} ) _{|\IND^\perp} \|^{1/s}_{\ell^2_n(\mu)\to \ell^2_n(\mu)} = \max (\bar \sigma_{n,\bp}(s),\sigma_{1,\bp} (s)) \leq 1 - \delta.
\end{equation}
We use Proposition \ref{stoop} for the walk $ X_t = \bar \varphi_n( \cX_t,x)$ with
$$T = \tau_m + s , \quad t=\lfloor (1+\delta)\log n /\fh- 2 \log\log n \rfloor \quad \text{ and } s=\lfloor \log \log n\rfloor .$$

For our choice of $s$, it follows from \eqref{eq:gdidie} that the third term in \eqref{en3} is smaller than $\gep/3$. 
It remains to prove that for $n$ sufficiently large
\begin{equation}\label{twopartsQ}
 \bbP[\tau_m>t]\le \gep/3 \quad \text{ and }  \quad \| P_{n,\bp}^s Q^m_n  ((x,u_0),\cdot)-\pi_n  \|_{\TV}\le \gep/3.
\end{equation}
where $Q_{n} = P_{n,\bq}$ is the Markov chain of the induced walk $X_{\tau_m}=  \bar \varphi_n(  \cX_{\tau_m},x)$ on $V_n \times [r]$. For the first inequality of \eqref{twopartsQ}, we choose $k(\delta)$ sufficiently large and it is a consequence of Lemma \ref{le:tau} and the law of large numbers.

\medskip

The second inequality of \eqref{twopartsQ} is obtained using spectral estimates for $Q_n = P_{n,\bq}$.
 The Cauchy-Schwarz inequality gives
\begin{equation}\label{spectrosQ}
 \| P_{n,\bp}^s Q_n^m ((x,u_0),\cdot)-\pi_n\|_{\TV}\le C_0 \sqrt{n }   \| Q_n^m  P_{n,\bp}^s   f  \|_{\ell_n^2(\mu)} ,
\end{equation}
with $C_0 = \sqrt{r / \min_i  \mu(i)}$ and   $f (y,v)= \delta_{(x,u_0) } (y,v) / \mu(v)   - 1 / n$.  Let $\Pi_{H}$ be the orthogonal projection in $\ell^2_n (\mu)$ onto a vector space $H$. We find
\begin{equation}\label{eq:gedgh}
  \|  Q^m_n P_{n,\bp} ^s   f \|_{\ell_n^2(\mu)}  \leq   \| Q^m_n  P_{n,\bp} ^s   \Pi_{H_r} f \|_{\ell_n^2(\mu)} + \|  Q^m_n   P_{n,\bp}^s \Pi_{H^\perp_r} f \|_{\ell_n^2(\mu)}.
\end{equation}
We now compute a spectral bound of the two terms on the right-hand side of \eqref{eq:gedgh}. We first observe that $\| f\|_{\ell^2 _n (\mu)} \leq 1 $ and $\| \Pi_{H_r} f\|_{\ell^2 _n (\mu)} \leq C/ \sqrt n  $  with $C = 1/ \sqrt{\min_i \mu(i)}$. Since $\langle f , \IND \rangle_{\mu} = 0$, we find from \eqref{eq:gdidie} and the fact that $Q_n$ is a contraction in $\ell^2_n (\mu)$,
\begin{equation}\label{zeboundQ0}
 \| Q^m_n  P_{n,\bp} ^s   \Pi_{H_r} f \|_{\ell_n^2(\mu)} \leq \|  P_{n,\bp} ^s    \Pi_{H_r} f \|_{\ell_n^2(\mu)} \leq \frac{C}{ \sqrt n} \sigma_{n,\bp}(s)^s   \leq  \frac{C}{ \sqrt n} (1 - \delta)^s .  
\end{equation}
We now give a bound of the second term on the right-hand side of \eqref{eq:gedgh}. From \eqref{eq:RDQ} and Proposition \ref{le:UnifU}, we have  for some constant $C$  
\begin{equation*}
\sigma_{ \bq} \le C k^{-1/2} (\log k)^{C}.
\end{equation*}
From \eqref{eq:SCQ} we deduce that for all $n$ large enough,
\begin{equation*}\label{zeboundQ}
 \bar \sigma_{n,\bq} \leq 2C k^{-1/2} (\log k)^{C}.
\end{equation*}
Since $\| f\|_{\ell^2 _n (\mu)} \leq 1  $, $P_{n,\bp} ^s \Pi_{H^\perp_r}  = \Pi_{H^\perp_r}P_{n,\bp}^s $ and $P_{n,\bp}$ is a contraction in $\ell^2_n(\mu)$, we deduce that
\begin{equation}\label{zeboundQ1}
\|  Q^m_n   P_{n,\bp}^s \Pi_{H^\perp_r} f \|_{\ell_n^2(\mu)} \leq  \bar \sigma_{n,\bq} ^m \| f \|_{\ell_n^2(\mu)} \leq 2C k^{-1/2} (\log k)^{C}
\end{equation}
 Equation \eqref{spectrosQ} together with \eqref{eq:gedgh}-\eqref{zeboundQ0}-\eqref{zeboundQ1} guaranties that $X_{\tau_m+s}$ is close to equilibrium in total variation. It concludes the proof of \eqref{twopartsQ}. \end{proof}

\section{Proof of Lemma \ref{le:simplelemma}}

Let us consider $\cA$ the set of finite words in the alphabet $[d]$, $\cB$ the set of words without repetition in $[d]$ and for a fixed involution $*$ on $[d]$, $\cC$ the set of finite words in which the patterns $i i^*$ and $i^*i$ do not appear.

Given $\alpha = (\alpha_i)_{i=1}^d$ a set of non-negative numbers in $[0,1)^d$, we define the function $F_{\alpha}(i_1. \dots. i_t)= \alpha_{i_1}\cdots \alpha_{i_t}$. 
We have immediately 
\begin{equation}
 \sum_{\cA}F_{\alpha}({\bf i})<\infty \Leftrightarrow  \sum_{[d]} \alpha_i <1.
\end{equation}
Now a word ${\bf i}\in \cA$ can be encoded by a word ${\bf j}\in\cB$ and
and a sequence $(n_t)_{t=1}^{|\bf j|}$ which counts how many time each letter is repeated. For this reason we have, given $\beta = (\beta_i)_{i=1}^d$,
$\sum_{\cA} F_{\beta}({\bf i})=  \sum_{\cB} F_{\beta'}({\bf i}),$
where $\beta_i'= \sum_{n\ge 1}\beta_i^n= \beta_i/(1-\beta_i)$. Hence taking $\beta_i=\alpha_i/(1+\alpha_i)$ we obtain 
\begin{equation}
 \sum_{\cB}F_{\alpha}({\bf i})<\infty \Leftrightarrow  \sum_{[d]} \frac{\alpha_i}{1+\alpha_i} <1.
\end{equation}
Finally to encode a word in ${\bf i}\in \cC$, we first consider a finite word ${\bf j}$ without repetition 
in $[d']$ where $d'$ is the number of conjugation classes for $*$. Let $\cB'$ be the set of such words $\bf j$. Then, to encode $\bf i$, we have to replace each of the letter of ${\bf j}$ by a pattern.
If the conjugation class $j \in [d']$ is a single element $\{i\}$ in $[d]$, there is only one possible pattern which is $i$. We thus define the weight of $j$ as  $\gamma_j := \alpha_i$.
Otherwise, the conjugation class $j \in [d']$ is a pair $\{i, i^*\}$. Then the possible patterns are $(i^*)^n$ and $i^n$, with $n\ge 1$. This gives a total weight  $\gamma_j := \alpha_i/(1-\alpha_i)+\alpha_{i^*} /(1-\alpha_{i^*})$. We thus have 
$\sum_{{\bf i}\in \cC} F_{\alpha}({\bf i})=   \sum_{{\bf j}\in \cB'} F_{\gamma}({\bf j})$.
In particular the sum is finite if and only if 
\begin{equation}
\sum_{j\in [d']} \frac{\gamma_j}{1+\gamma_j}=\sum_{i\in [d]}\left(
\frac{\alpha_i}{1+\alpha_i}\ind_{\{i=i^*\}}+\frac{1}{2} \frac{\alpha_i (1-\alpha_i)^{-1}+\alpha_{i^*}(1-\alpha_{i^*})^{-1} }{1+ \alpha_{i}(1-\alpha_i)^{-1}+\alpha_{i^*}(1-\alpha_{i^*})^{-1}} \ind_{\{i\ne i^*\}} \right)<1.
\end{equation}
This is the required statement. \qed

\bibliographystyle{abbrv}
\bibliography{bib}

\begin{thebibliography}{10}

\bibitem{MR0442698}
C.~A. Akemann and P.~A. Ostrand.
\newblock Computing norms in group {$C\sp*$}-algebras.
\newblock {\em Amer. J. Math.}, 98(4):1015--1047, 1976.

\bibitem{aldous1983mixing}
D.~Aldous.
\newblock Random walks on finite groups and rapidly mixing {M}arkov chains.
\newblock In {\em Seminar on probability, {XVII}}, volume 986 of {\em Lecture
  Notes in Math.}, pages 243--297. Springer, Berlin, 1983.

\bibitem{aldous1986shuffling}
D.~Aldous and P.~Diaconis.
\newblock {Shuffling cards and stopping times}.
\newblock {\em American Mathematical Monthly}, pages 333--348, 1986.

\bibitem{MR875835}
N.~Alon.
\newblock Eigenvalues and expanders.
\newblock {\em Combinatorica}, 6(2):83--96, 1986.
\newblock Theory of computing (Singer Island, Fla., 1984).

\bibitem{MR2348845}
N.~Alon, I.~Benjamini, E.~Lubetzky, and S.~Sodin.
\newblock Non-backtracking random walks mix faster.
\newblock {\em Commun. Contemp. Math.}, 9(4):585--603, 2007.

\bibitem{MR1883559}
A.~Amit and N.~Linial.
\newblock Random graph coverings. {I}. {G}eneral theory and graph connectivity.
\newblock {\em Combinatorica}, 22(1):1--18, 2002.

\bibitem{MR2216470}
A.~Amit and N.~Linial.
\newblock Random lifts of graphs: edge expansion.
\newblock {\em Combin. Probab. Comput.}, 15(3):317--332, 2006.

\bibitem{MR0324741}
A.~Avez.
\newblock Entropie des groupes de type fini.
\newblock {\em C. R. Acad. Sci. Paris S\'er. A-B}, 275:A1363--A1366, 1972.

\bibitem{MR3650406}
R.~Basu, J.~Hermon, and Y.~Peres.
\newblock Characterization of cutoff for reversible {M}arkov chains.
\newblock {\em Ann. Probab.}, 45(3):1448--1487, 2017.

\bibitem{MR3650414}
A.~Ben-Hamou and J.~Salez.
\newblock Cutoff for nonbacktracking random walks on sparse random graphs.
\newblock {\em Ann. Probab.}, 45(3):1752--1770, 2017.

\bibitem{MR3758735}
N.~Berestycki, E.~Lubetzky, Y.~Peres, and A.~Sly.
\newblock Random walks on the random graph.
\newblock {\em Ann. Probab.}, 46(1):456--490, 2018.

\bibitem{MR2408585}
S.~Blach\`ere, P.~Ha\"{\i}ssinsky, and P.~Mathieu.
\newblock Asymptotic entropy and {G}reen speed for random walks on countable
  groups.
\newblock {\em Ann. Probab.}, 36(3):1134--1152, 2008.

\bibitem{bordenaveCAT}
C.~Bordenave.
\newblock A new proof of {F}riedman{'}s second eigenvalue theorem and its
  extension to random lifts.
\newblock arXiv:1502.04482, 2015.

\bibitem{MR3773804}
C.~Bordenave, P.~Caputo, and J.~Salez.
\newblock Random walk on sparse random digraphs.
\newblock {\em Probab. Theory Related Fields}, 170(3-4):933--960, 2018.

\bibitem{BC18}
C.~Bordenave and B.~Collins.
\newblock Eigenvalues of random lifts and polynomials of random permutation
  matrices.
\newblock {\em Ann. of Math. (2)}, 190(3):811--875, 2019.

\bibitem{MR0338272}
P.~Cartier.
\newblock {\em Harmonic analysis on trees}.
\newblock Amer. Math. Soc., Providence, R.I., 1973.

\bibitem{CECCHERINISILBERSTEIN2004735}
T.~Ceccherini-Silberstein, F.~Scarabotti, and F.~Tolli.
\newblock Weighted expanders and the anisotropic {A}lon–{B}oppana theorem.
\newblock {\em European Journal of Combinatorics}, 25(5):735 -- 744, 2004.

\bibitem{MR3666050}
I.~Chatterji.
\newblock Introduction to the rapid decay property.
\newblock In {\em Around {L}anglands correspondences}, volume 691 of {\em
  Contemp. Math.}, pages 53--72. Amer. Math. Soc., Providence, RI, 2017.

\bibitem{C-K}
G.~Conchon-Kerjan.
\newblock Cutoff for random lifts of weighted graphs.
\newblock \href{https://arxiv.org/abs/1908.02898}{arXiv:1908.02898}.

\bibitem{diaconis1996cutoff}
P.~Diaconis.
\newblock The cutoff phenomenon in finite {M}arkov chains.
\newblock {\em Proc. Nat. Acad. Sci. U.S.A.}, 93(4):1659--1664, 1996.

\bibitem{diaconis1981generating}
P.~Diaconis and M.~Shahshahani.
\newblock {Generating a random permutation with random transpositions}.
\newblock {\em Probability Theory and Related Fields}, 57(2):159--179, 1981.

\bibitem{MR1219707}
A.~Fig\`a-Talamanca and T.~Steger.
\newblock Harmonic analysis for anisotropic random walks on homogeneous trees.
\newblock {\em Mem. Amer. Math. Soc.}, 110(531):xii+68, 1994.

\bibitem{MR1978881}
J.~Friedman.
\newblock Relative expanders or weakly relatively {R}amanujan graphs.
\newblock {\em Duke Math. J.}, 118(1):19--35, 2003.

\bibitem{MR824708}
P.~Gerl and W.~Woess.
\newblock Local limits and harmonic functions for nonisotropic random walks on
  free groups.
\newblock {\em Probab. Theory Relat. Fields}, 71(3):341--355, 1986.

\bibitem{MR1502972}
J.~Geronimus.
\newblock On a set of polynomials.
\newblock {\em Ann. of Math. (2)}, 31(4):681--686, 1930.

\bibitem{MR1802431}
R.~I. Grigorchuk and A.~\.{Z}uk.
\newblock On the asymptotic spectrum of random walks on infinite families of
  graphs.
\newblock In {\em Random walks and discrete potential theory ({C}ortona,
  1997)}, Sympos. Math., XXXIX, pages 188--204. Cambridge Univ. Press,
  Cambridge, 1999.

\bibitem{MR520930}
U.~Haagerup.
\newblock An example of a nonnuclear {$C^{\ast} $}-algebra, which has the
  metric approximation property.
\newblock {\em Invent. Math.}, 50(3):279--293, 1978/79.

\bibitem{MR1784419}
U.~Haagerup and F.~Larsen.
\newblock Brown's spectral distribution measure for {$R$}-diagonal elements in
  finite von {N}eumann algebras.
\newblock {\em J. Funct. Anal.}, 176(2):331--367, 2000.

\bibitem{MR3693771}
J.~Hermon.
\newblock Cutoff for {R}amanujan graphs via degree inflation.
\newblock {\em Electron. Commun. Probab.}, 22:Paper No. 45, 10, 2017.

\bibitem{MR3765366}
J.~Hermon.
\newblock A technical report on hitting times, mixing and cutoff.
\newblock {\em ALEA Lat. Am. J. Probab. Math. Stat.}, 15(1):101--120, 2018.

\bibitem{MR704539}
V.~A. Ka\u\i\~manovich and A.~M. Vershik.
\newblock Random walks on discrete groups: boundary and entropy.
\newblock {\em Ann. Probab.}, 11(3):457--490, 1983.

\bibitem{MR112053}
H.~Kesten.
\newblock Full {B}anach mean values on countable groups.
\newblock {\em Math. Scand.}, 7:146--156, 1959.

\bibitem{MR0109367}
H.~Kesten.
\newblock Symmetric random walks on groups.
\newblock {\em Trans. Amer. Math. Soc.}, 92:336--354, 1959.

\bibitem{MR1832436}
F.~Ledrappier.
\newblock Some asymptotic properties of random walks on free groups.
\newblock In {\em Topics in probability and {L}ie groups: boundary theory},
  volume~28 of {\em CRM Proc. Lecture Notes}, pages 117--152. Amer. Math. Soc.,
  Providence, RI, 2001.

\bibitem{MR1844214}
F.~Lehner.
\newblock On the computation of spectra in free probability.
\newblock {\em J. Funct. Anal.}, 183(2):451--471, 2001.

\bibitem{MR3726904}
D.~A. Levin, Y.~Peres, and E.~L. Wilmer.
\newblock {\em Markov chains and mixing times}.
\newblock American Mathematical Society, Providence, RI, 2017.
\newblock Second edition of [ MR2466937], With a chapter on ``Coupling from the
  past'' by James G. Propp and David B. Wilson.

\bibitem{MR2536865}
L.~Lov\'{a}sz and M.~D. Plummer.
\newblock {\em Matching theory}.
\newblock AMS Chelsea Publishing, Providence, RI, 2009.
\newblock Corrected reprint of the 1986 original [MR0859549].

\bibitem{lubetzky}
E.~Lubetzky.
\newblock private communication.

\bibitem{MR3558308}
E.~Lubetzky and Y.~Peres.
\newblock Cutoff on all {R}amanujan graphs.
\newblock {\em Geom. Funct. Anal.}, 26(4):1190--1216, 2016.

\bibitem{MR2667423}
E.~Lubetzky and A.~Sly.
\newblock Cutoff phenomena for random walks on random regular graphs.
\newblock {\em Duke Math. J.}, 153(3):475--510, 2010.

\bibitem{MR3585560}
J.~A. Mingo and R.~Speicher.
\newblock {\em Free probability and random matrices}, volume~35 of {\em Fields
  Institute Monographs}.
\newblock Springer, New York; Fields Institute for Research in Mathematical
  Sciences, Toronto, ON, 2017.

\bibitem{MR2679612}
B.~Mohar.
\newblock A strengthening and a multipartite generalization of the
  {A}lon-{B}oppana-{S}erre theorem.
\newblock {\em Proc. Amer. Math. Soc.}, 138(11):3899--3909, 2010.

\bibitem{MR1554815}
J.~Petersen.
\newblock Die {T}heorie der regul\"{a}ren graphs.
\newblock {\em Acta Math.}, 15(1):193--220, 1891.

\bibitem{MR1401692}
G.~Pisier.
\newblock A simple proof of a theorem of {K}irchberg and related results on
  {$C^*$}-norms.
\newblock {\em J. Operator Theory}, 35(2):317--335, 1996.

\bibitem{MR2377835}
S.~Sodin.
\newblock Random matrices, nonbacktracking walks, and orthogonal polynomials.
\newblock {\em J. Math. Phys.}, 48(12):123503, 21, 2007.

\end{thebibliography}

\end{document}